\documentclass[a4paper,10pt,leqno]{amsart}
\usepackage{amsmath,amsfonts,amssymb,amsthm,epsfig,epstopdf,url,array}
\usepackage{amsthm}
\usepackage{mathrsfs}
\usepackage{epsfig}
\usepackage{graphicx}

\makeatletter
\newcommand{\bigperp}{%
  \mathop{\mathpalette\bigp@rp\relax}%
  \displaylimits
}

\newcommand{\bigp@rp}[2]{%
  \vcenter{
    \m@th\hbox{\scalebox{\ifx#1\displaystyle2.1\else1.5\fi}{$#1\perp$}}
  }%
}
\makeatother

\newtheorem{thm}{Theorem}[section]
\newtheorem{prp}[thm]{Proposition}
\newtheorem{cor}[thm]{Corollary}
\newtheorem{lem}[thm]{Lemma}
\newtheorem{rmk}[thm]{Remark}

\numberwithin{equation}{section}

\newcommand{\N}{{\mathbb{N}}}
\newcommand{\Z}{{\mathbb{Z}}}
\newcommand{\z}{{\mathbb{Z}}}

\newcommand{\ord}{\text{ord}}

\newcommand{\nequiv}{\not\equiv}

\newcommand{\rank}{\operatorname{rank}}

\newcommand{\ceil}[1]{\left\lceil #1 \right\rceil}

\title[A finiteness theorem for universal $m$-gonal forms]{A finiteness theorem for universal $m$-gonal forms}

\author{BYEONG MOON KIM and Dayoon Park}
\address{Department of Mathematics, Kangnung National University, Kangnung, 210-702, Korea}
\email{kbm@kangnung.ac.kr}
\thanks{}

\address{Department of Mathematics, The University of Hong Kong, Hong Kong}
\email{pdy1016@hku.hk}
\thanks{}

\begin{document}
\maketitle

\begin{abstract}
In this paper, we study the set of positive integers that characterize the universality of $m$-gonal form.
\end{abstract}

\section{introduction and statement of results}

In 1993, Conway and Schneeberger announced the {\it $15$-theorem} stating that a classical quadratic form is universal if it represents all positive integers up to $15$.
This stunning criterion makes to readily check the universality of classic quadratic forms.
They also gave the {\it $290$-conjecture} that every quadratic form, regardless of being classical or not, is universal if it represents all positive integer up to $290$.
Baghava and Hanke \cite{290} proved this conjecture in 2000.
Baghava also showed that for every $S \subset \mathbb N$ there is a finite subset $S_0$ of $S$ such that if a quadratic form represents all elements of $S_0$, then it represents all numbers in $S$.
This criterion is generalized to arbitrary rank by the first author, Kim, and Oh.
In other words, for any set $S$ of quadratic forms with a fixed rank $n$, there is a finite subset $S_0$ of $S$ such that every quadratic form representing all members of $S_0$ represents all elements of $S$.

In the same manner with the 15-theorem, one may consider similar theorem on $m$-gonal form.
The {\it $x$-th $m$-gonal number} 
\begin{equation} \label{polynum} P_m(x)=\frac{(m-2)x^2-(m-4)x}{2}\end{equation}
 is defined as the total number of dots to constitute a regular $m$-gon with $x$ dots for each side for $m \ge 3$ with $x \in \N$.
On the other hand, we may generalize the $m$-gonal number (\ref{polynum}) by allowing the variable $x$ to be not only positive integer but also non-positive integer too.
In this paper, we call a weighted sum of generalized $m$-gonal numbers
\begin{equation} \label{polyform}\Delta_{m,\mathbf a}(\mathbf x)=\sum \limits_{i=1}^n a_iP_m(x_i)\end{equation}
with $a_i \in \mathbb N$ and $x_i \in \mathbb Z$ as {\it $m$-gonal form}.
In \cite{BJ}, Kane and Liu claimed that there always exists a (unique, minimal) $\gamma_m \in \mathbb N$ such that if an $m$-gonal form 
$\Delta_{m,\mathbf a}(\mathbf x)$ represents every positive integer up to $\gamma_m$, then the $m$-gonal form $\Delta_{m,\mathbf a}(\mathbf x)$ is universal.
There are some examples for $\gamma_m$ which are concretely calculated for some small $m$'s.
Bosma and Kane \cite{BK} showed $\gamma_3=8$. The $15$-Theorem deduces $\gamma_4=15$.
Ju \cite{J} showed $\gamma_5=109$.
Since any hexagonal number is a triangular number and any triangular number is a hexagonal number, one can get $\gamma_3=\gamma_6=8$.
Ju and Oh \cite{JO} proved that $\gamma_8=60$.

An interesting question about the growth of $\gamma_m$ as a function of $m$ was firstly suggested by Kane and Liu \cite{BJ} who proved $m-4 \le \gamma_m \ll m^{7+\epsilon}$.
In this paper, we improve the Kane and Liu's result by showing that the growth of $\gamma_m$ is exactly linear on $m$.
On the other words, we show the following theorem
\begin{thm} \label{main}
There exists an absolute constant $C>0$ such that $$m-4 \le \gamma_m \le C(m-2)$$
for any $m \ge 3$.
\end{thm}

Although in \cite{BJ}, they used an analytic method, we use only the arithmetic theory of quadratic forms.
Especially we concentrate on the study of a system 
$$\begin{cases}a_1x_1^2+\cdots+a_nx_n^2 \equiv t \pmod{s}\\ a_1x_1+\cdots+a_nx_n=r \end{cases}$$
 of a quadratic congruence and a linear equation by applying both local theory and global theory of quadratic forms.

\section{preliminary}

Before we move on, we introduce our languages and some notations.
We adopt the language of arithmetic theory of quadratic forms.
The term {\it lattice} always refer to an integral quadratic $\mathbb Z$-lattice on a positive definite quadratic space over $\mathbb Q$, i.e., $\mathbb Z$-module $L$ consisting a quadratic map $Q$ on it and the corresponding bilinear symmetric mapping $B(x,y)=\frac{1}{2}(Q(x+y)-Q(x)-Q(y))$.
We write the (quadratic) $\mathbb Z$-lattice as $(L,Q)$ and if there is no confusion on the emploied quadratic map $Q$, then we simply write the $\mathbb Z$-lattice as $L$.
We say an integer $n \in \Z$ is {\it represented} by $L$ if there is an element $\mathbf v \in L $ for which $Q(\mathbf v)=n$.  
In this paer, without loss of generality, we always assume that $a_i \le a_{i+1}$ in (\ref{polyform}).
We say an integer $n \in \Z$ is {\it represented} by $\Delta_{m,\mathbf a}(\mathbf x)$ if there is an element $\mathbf v \in \Z^n $ for which $\Delta_{m,\mathbf a}(\mathbf v)=n$.  

We denote $\mathbf e_i$ as the vector with $1$ in the $i$-th coordinate and $0$'s elsewhere. 
We put $\alpha_n = \sum \limits_{i=1}^{n}\mathbf e_i \in \mathbb Z^n$ and especially when $n=10$, simply write $\alpha:=\alpha_{10} \in \mathbb Z^{10}$.

For $\mathbf a \in \N^n$, we define an $n$-ary diagonal quadratic $\mathbb Z$-lattice {\it $(\mathbb Z^n, Q_{\mathbf a})$} by $$Q_{\mathbf a}(x_1,\cdots, x_n):=a_1x_1^2+\cdots+a_nx_n^2$$ and write the corresponding symmetric bilinear mapping as $B_{\mathbf a}$, i.e., 
$$B_{\mathbf a}(\mathbf x,\mathbf y):=\sum \limits_{i=1}^na_ix_iy_i.$$
And we define an $(n-1)$-ary $\mathbb Z$-sublattice $(L_{\mathbf a},Q_{\mathbf a})$ of $(\mathbb Z^n,Q_{\mathbf a})$ by a hyperplane $$L_{\mathbf a}:=\{(x_1,\cdots,x_n) \in \mathbb Z^n| B_{\mathbf a}(\alpha_n,\mathbf x)=a_1x_1+\cdots+a_nx_n=0\}.$$
For a vector $\beta=(\beta_1,\cdots,\beta_n)\in \mathbb Z^n$, we define an $(n-2)$-ary $\mathbb Z$-sublattice $(L_{\mathbf a,\beta},Q_{\mathbf a})$ of $(L_{\mathbf a},Q_{\mathbf a})$ by $$L_{\mathbf a,\beta}:=\{(x_1,\cdots, x_n) \in L_{\mathbf a} | B_{\mathbf a}(\beta,\mathbf x)=a_1\beta_1 x_1+\cdots + a_n\beta_n x_n=0\}.$$
Any unexplained terminology and notation can be found in \cite{O}.

\vskip 1em
The following basic facts for $\gamma_m$ are known.

\begin{lem} \label{esc}
\begin{itemize}

\item[(i)] $\gamma_m$ always exists and is finite for any $m \ge 3$.
\item[(ii)] $\gamma_m \ge m-4$.
\item[(iii)] 
In order for an $m$-gonal form $\Delta_{m,\mathbf a}(x_1,\cdots,x_n)= \sum\limits_{i=1}^n a_iP_m(x_i)$ with $a_i \le a_{i+1}$ to represent all positive integers up to $m-4$, the followings should be satisfied

\begin{equation} \label{gm}
\begin{cases}
a_1=1 & \\
a_{i+1}\le a_1+\cdots+a_i+1 & \text{for all $i$ such that $a_1+\cdots+a_i<m-4$}\\
\sum\limits_{i=1}^n a_i \ge m-4 & \\
\end{cases}
\end{equation}
\end{itemize}
\end{lem}
\begin{proof}
(i) is due to Kane and Liu \cite{BJ}.\\
(ii) is due to Guy \cite{G}.\\
(iii) directly follows from the fact that the smallest nonzero $m$-gonal number is $P_m(-1)=m-3$ except $P_m(1)=1$.

\end{proof}
\begin{rmk} \label{r0}
From $a_1=1$, inductively constructed $a_k$ under the condition $a_{i+1} \le a_1+\cdots + a_i +1$ for all $i \le k-1$ can not exceed $2^{k-1}$. So by Lemma \ref{esc}, if $m$ is large, in order for an $m$-gonal form $\Delta_{m,\mathbf a}(\mathbf x)$ to represent every positive integer up to only $m-4$, it is needed quite a lot of components more precisely, there is needed at least $\ceil{\log_2 (m-3)}$ components.
\end{rmk}

\section{Motivation and partial results}
The following well known local-to-global theorem in the arithmetic theory of quadratic form ensures for being represented of sufficiently large integer which is locally represented.
\begin{thm} \label{hkk}
Let $(L,Q)$ be a quadratic $\mathbb Z$-lattice of $\rank (L) \ge 5$.
Then there is a constant $c(L)>0$ such that $L$ represents any integer $n$, provided that 
$$\begin{cases}
n \ge c(L)\\
L \otimes \mathbb Z_p\text{ represents }n \text{ at every prime }p.
\end{cases}$$
\end{thm}
\begin{proof}
See \cite{T}.
\end{proof}

\vskip 1em

An $m$-gonal form $\Delta_{m,\mathbf a}(\mathbf x)=\sum \limits_{i=1}^n a_iP_m(x_i)$ can be expressed as \begin{equation}\label{b} \Delta_{m,\mathbf a} (\mathbf x)=\frac{m-2}{2}\{Q_{\mathbf a}(\mathbf x)-B_{\mathbf a}(\alpha_n,\mathbf x)\}+B_{\mathbf a}(\alpha_n,\mathbf x). \end{equation} 
With the expression (\ref{b}), we particularly investigate the cases that the linear form $B_{\mathbf a}(\alpha_n, \mathbf x)$ is a specified integers. 
First of all, in Lemma \ref{sf}, we consider the case $B_{\mathbf a}(\alpha_n,\mathbf x)=0$.

\vskip 1em

\begin{lem} \label{sf}
For $6$-tuple $\mathbf a:=(a_1,\cdots, a_6) \in \mathbb N^6$, there is a constant $s(\mathbf a) \in \mathbb N$ for which the senary $m$-gonal form 
$$\Delta_{m,\mathbf a}(\mathbf x)=\sum_{i=1}^6a_iP_m(x_i) $$
 represents every multiple of $s(\mathbf a)\cdot(m-2)$ for any $m \ge 3$.
\end{lem}
\begin{proof}
If $\mathbf x \in L_{\mathbf a}$, then since $B_{\mathbf a}(\alpha_6,\mathbf x)=0$, we have that $$\Delta_{m,\mathbf a}(\mathbf x)=\frac{m-2}{2}Q_{\mathbf a}(\mathbf x)$$ by (\ref{b}).
Since $\rank(L_{\mathbf a})=5$, the quadratic $\Z$-lattice $(L_{\mathbf a},Q_{\mathbf a})$ is isotropic over $\mathbb Z_p$ for all primes $p$.
So we may get that for each prime $p$, there is a minimal integer $e(p)$ for which $L_{\mathbf a} \otimes \mathbb Z_p$ represents every $p$-adic integer in $p^{e(p)}\mathbb Z_p$ with $e(p)=0$ for almost all primes $p$.
Since the quadratic $\z$-lattice $L_{\mathbf a}$ locally represents all the multiples of $e:=\prod p^{e(p)} (\in \N)$,
by Theorem \ref{hkk}, there is a constant $2s(\mathbf a)$ (with $e | 2s(\mathbf a)$) for which the quadratic $\mathbb Z$-lattice $(L_{\mathbf a},Q_{\mathbf a})$ represents all the multiples of $2s(\mathbf a)$.
Which induces that the $m$-gonal form $\Delta_{m,\mathbf a}(\mathbf x)$ represents every multiple of $s(\mathbf a)\cdot(m-2)$ when $\mathbf x $ runs through $L_{\mathbf a}(\subset \z^6)$ from (\ref{b}).
\end{proof}

\vskip 0.8em

\begin{rmk} \label{rmk s}
For sufficiently large $m >2^{16}+3$, in order for an $m$-gonal form $\Delta_{m,\mathbf a}(\mathbf x)$ to represent every positive integer up to $m-4$, it is required at least $16( \le \log_2 (m-4))$ components and the tuple of its first $16$ coefficients must coincide with one of $16$-tuples in the finite set $$A:=\{(a_1,\cdots,a_{16})\in \mathbb N^{16}| a_1=1, a_i \le a_{i+1}\le a_1+\cdots+a_i+1 \text{ for all }1\le i \le 15\}$$
by Lemma \ref{esc}.

Let $s$ be the least common multiple of all $s(a_{j_1},\cdots, a_{j_6})$'s where $(a_{j_1},\cdots, a_{j_6})$ runs through all the $6$-subtuples of elements in $A$ $$\{(a_{j_1},\cdots, a_{j_6}) \in \mathbb N^6|(a_1,\cdots, a_{16}) \in A \text{ with }1\le j_1 < \cdots < j_6 \le 16\}.$$
Then we may obtain following Corollary.
\end{rmk}

\begin{cor} \label{6}
For $m>2^{16}+3$, let $\Delta_{m,\mathbf a}(\mathbf x)= \sum\limits_{i=1}^n a_iP_m(x_i)$ be an $m$-gonal form which represents every positive integer up to $m-4$. 
Then for any senary sub $m$-gonal form $$a_{j_1}P_m(x_{j_1})+a_{j_2}P_m(x_{j_2})+\cdots+a_{j_6}P_m(x_{j_6})$$ of $\Delta_{m,\mathbf a}(\mathbf x)$ where $1 \le j_1<\cdots <j_6 \le 16$ represents all the multiples of $s(m-2)$ where $s$ is the constant given in Remark \ref{rmk s}.
\end{cor}
\begin{proof}
One may use Lemma \ref{sf}.
\end{proof}

\begin{rmk}
For an $m$-gonal form $\Delta_{m,\mathbf a}(\mathbf x)=\sum _{i=1}^na_iP_m(x_i)$ with $(a_1,\cdots ,a_{16}) \in A$,
its arbitrary senary subform 
$$a_{j_1}P_m(x_{j_1})+a_{j_2}P_m(x_{j_2})+\cdots+a_{j_6}P_m(x_{j_6})$$ satisfying $1 \le j_1<\cdots <j_6 \le 16$ would represent every multiple of $s(m-2)$ by Corollary \ref{6}.
Therefore the existence of $1 \le i_1<\cdots<i_{10} \le 16$ for which the $(n-6)$-ary subform 
$$a_{i_1}P_m(x_{i_1})+\cdots +a_{i_{10}}P_m(x_{i_{10}})+\sum_{i=17}^na_iP_m(x_i)$$
of $\Delta_{m,\mathbf a}(\mathbf x)$ represents complete residues modulo $s(m-2)$ in some interval $[0,N]$
would be enough to yield that the $m$-gonal form $\Delta_{m,\mathbf a}(\mathbf x)$ represents every positive integer greater than $N$.

For $(a_1,\cdots,a_{16}) \in A$, if there are $10$-subtuple
\begin{equation} (a_{i_1},\cdots,a_{i_{10}}) \end{equation} 
of $(a_1,\cdots, a_{16})$ with $1 \le i_1<\cdots<i_{10} \le 16 $ and a constant 
\begin{equation} C_{(a_1,\cdots,a_{16})}>0 \end{equation} 
which is dependent only on $(a_1,\cdots,a_{16})$ for which the $(n-6)$-ary subform
$$a_{i_1}P_m(x_{i_1})+\cdots +a_{i_{10}}P_m(x_{i_{10}})+\sum_{i=17}^na_iP_m(x_i)$$
of any $m$-gonal form $\Delta_{m,\mathbf a}(\mathbf x)$ (having its first $16$ coefficients as $(a_1,\cdots ,a_{16})$) which represents every positive integer up to $m-4 (\le \gamma_m)$ represents complete residues modulo $s(m-2)$ in $[0,C_{(a_1,\cdots ,a_{16})}(m-2)]$, then we may say that ``the universality of any $m$-gonal form having its first $16$ coefficients as $(a_1,\cdots ,a_{16})$ is characterized by the representability of every positive integer up to $C_{(a_1,\cdots ,a_{16})}(m-2)$''.
\end{rmk}

\vskip 1em

\begin{lem} \label{lu}
For $(a_1,\cdots,a_{16}) \in A$, if there are $10$-subtuple $$\mathbf a_2:=(a_{i_1},\cdots,a_{i_{10}})$$ of $(a_1,\cdots,a_{16})$ with $1 \le i_1<\cdots<i_{10} \le 16$ and vector $$\beta=(\beta_1,\cdots,\beta_{10}) \in \mathbb Z^{10}$$ satisfying $B_{\mathbf a_2}(\alpha,\beta)=a_{i_1}\beta_1+\cdots+a_{i_{10}}\beta_{10}=1$ for which the octonary quadratic $\z$-lattice 
\begin{equation}
\text{($L_{\mathbf a_2, \beta},Q_{\mathbf a_2})$ is locally even universal,}
\end{equation} 
then there is a constant $$C_{(a_1,\cdots, a_{16})} >0 $$ for which 
having its first $16$ coefficients as $(a_1,\cdots ,a_{16}) \in A$ $m$-gonal form  
$$\Delta_{m,\mathbf a}(\mathbf x)=\sum_{i=1}^{16}a_iP_m(x_i)+\sum_{i=17}^{n}a_iP_m(x_i)$$ which represents every postivie integer up to $m-4$ represents every integer greater than $C_{(a_1,\cdots,a_{16})}(m-2)$. 
\end{lem}

\begin{proof}
In virtue of Corollary \ref{6}, it would be enough to show the existence of a constant $C_{(a_1,\cdots,a_{16})}>0$ for which the $(n-6)$-ary sub $m$-gonal form
$$a_{i_1}P_m(x_{i_{1}})+\cdots+a_{i_{10}}P_m(x_{i_{10}})+\sum\limits_{i=17}^n a_iP_m(x_i)$$ represents complete residues modulo $s(m-2)$ in $[0,C_{(a_1,\cdots,a_{16})}(m-2)]$.
In other words, we show that for each residue $N$ modulo $s(m-2) \Z$, there exists an integer solution $(x_{i_1},\cdots, x_{i_{10}}, x_{17},\cdots,x_n) \in \Z^{n-6}$ for 
$$\begin{cases}\Delta_{m,\mathbf a_2}(x_{i_1},\cdots, x_{i_{10}})+\sum \limits_{i=17}^na_iP_m(x_i) \equiv N \pmod{s(m-2)}\\
\Delta_{m,\mathbf a_2}(x_{i_1},\cdots, x_{i_{10}})+\sum \limits_{i=17}^na_iP_m(x_i) \le C_{(a_1,\cdots,a_{16})}(m-2).\end{cases}$$

For each $0\le r \le m-3$, we may take a vector $\mathbf x(r)=(x_1(r),\cdots, x_n(r)) \in \mathbb Z^n$ such that $$\Delta_{m,\mathbf a}(\mathbf x(r))=r$$ by assumption. 
We put $$\begin{cases}r_1:=\sum \limits_{i=1}^{16}a_iP_m(x_i(r)) \\ r_2:=\sum \limits_{i=17}^{n}a_iP_m(x_i(r)).\end{cases}$$
Since the smallest non-zero $m$-gonal number is $P_m(1)=1$ and the next one is $P_m(-1)=m-3$, if $m>2^{16}-3$, then we have that
$$0 \le r_1 \le a_1+\cdots+a_{16}\le 2^{16}-1 \ \quad \text{ or } \ \quad r_1=m-3 \equiv -1 \pmod{m-2}.$$
Therefore $r_1$ may be congruent to $-1,0,1,2,\cdots, 2^{16}-2,$ or $ 2^{16}-1$ modulo $m-2$.

For $\mathbf x \in L_{\mathbf a_2,\beta}$ and $r_1 \in \mathbb Z$, we have 
\begin{equation}\label{eq2} \Delta_{m,\mathbf a_2}(\mathbf x + r_1 \beta)=\frac{m-2}{2}\left\{Q_{\mathbf a_2}(\mathbf x)+r_1^2Q_{\mathbf a_2}(\beta)-r_1\right\}+r_1.\end{equation}
Since the quadratic $\z$-lattice $L_{\mathbf a_2,\beta}$ with $\rank(L_{\mathbf a,\beta})=8$ is locally even universal, $L_{\mathbf a_2,\beta}$ represents every sufficiently large even integer  by Theorem \ref{hkk}. For $$r_1':=\begin{cases}-1 & \text{ if } r_1=m-3 \\ r_1 & \text{ otherwise, }\end{cases}$$
with $r_1\equiv r_1' \pmod{m-2}$ and  $-1\le r_1' \le 2^{16}-1$ and $0\le t \le s-1$, we may take a vector $\mathbf x (t,r_1') \in L_{\mathbf a_2,\beta}$ such that 
$$Q_{\mathbf a_2}(\mathbf x (t,r_1')) \equiv 2t - r_1'^2Q_{\mathbf a_2}(\beta)+|r_1'| \pmod{2s}$$
since $L_{\mathbf a_2,\beta}$ represents every sufficiently large even integer.
By applying this to (\ref{eq2}), we obtain that
$$\Delta_{m,\mathbf a_2}(\mathbf x (t,r_1')+r_1'\beta) \equiv t(m-2)+r_1 \pmod{s(m-2)}$$
and so for $0 \le t \le s-1$ and $0 \le r \le m-3$, 
\begin{align*}
\Delta_{m,\mathbf a_2}(\mathbf x (t,r_1')+r_1'\beta)+\sum \limits_{i=17}^na_iP_m(x_i(r)) & \equiv \  t(m-2)+r_1 \ + \ r_2 \\
& = \ t(m-2)+r \pmod{s(m-2)}.
\end{align*}
Namely, $\Delta_{m,\mathbf a_2}(x_{i_1},\cdots, x_{i_{10}})+\sum_{i=17}^n a_iP_m(x_i)$ represents complete residues modulo $s(m-2)$.
Now let $C_{(a_1,\cdots,a_{16})}$ be the constant 
$$\max  \left\{\frac{Q_{{\mathbf a}_2}(\mathbf x (t,r'))+(r')^2Q_{{\mathbf a}_2}(\beta)-r'}{2}|0 \le t \le s-1, -1 \le r' \le 2^{16}-1\right\}+1.$$
Then for each $0\le t \le s-1$ and $0 \le r \le m-3$, we finally obtain 
$$
\Delta_{m,\mathbf a_2}(\mathbf x (t,r_1') +r_1'\beta)+\sum \limits_{i=17}^na_iP_m(x_i(r)) \equiv t(m-2)+r \pmod{s(m-2)}$$
and 
\begin{align*}
&\Delta_{m,\mathbf a_2}(\mathbf x(t,r_1')+r_1'\beta)+\sum \limits_{i=17}^na_iP_m(x_i(r))  \\
=&\frac{m-2}{2}(Q_{\mathbf a_2}(\mathbf x(t,r_1'))+r_1'^2Q_{\mathbf a_2}(\beta)-r_1')+r_1+r_2\\
\le &(m-2) \frac{Q_{\mathbf a_2}(\mathbf x(t,r_1'))+r_1'^2Q_{\mathbf a_2}(\beta)-r_1'}{2}+(m-2)\\
\le &C_{(a_1,\cdots,a_{16})}(m-2).
\end{align*}
By the construction, $C_{(a_1,\cdots,a_{16})}$ is clearly dependent only on $(a_1,\cdots,a_{16})$.
This completes the proof.
\end{proof}

\begin{rmk}
Based on Lemma \ref{lu}, what should we do now is that to show that for each candidate $(a_1,\cdots, a_{16}) \in A$, we could take $10$-subtuple $$\mathbf a_2:=(a_{i_1},\cdots, a_{i_{10}})$$ with $1\le i_1 < \cdots < i_{10}\le 16$ of $\mathbf a=(a_1,\cdots, a_n)$ and a vector 
$$\beta=(\beta_1,\cdots,\beta_{10}) \in \mathbb Z^{10}$$ 
satisfying $B_{\mathbf a_2}(\alpha,\beta)=a_{i_1}\beta_1+\cdots+a_{i_{10}}\beta_{10}=1$ for which 
\begin{equation}
\text{($L_{\mathbf a_2, \beta},Q_{\mathbf a_2})$ is locally even universal.}
\end{equation}

The sublattice $(L_{\mathbf a_2, \beta},Q_{\mathbf a_2})$ of $(\mathbb Z^{10}, Q_{\mathbf a_2})$ is orthogonal complement $(\mathbb Z \alpha+\mathbb Z \beta)^{\perp}=\{\mathbf x\in \mathbb Z^{10}|B_{\mathbf a_2}(\mathbf x,\alpha)=B_{\mathbf a_2}(\mathbf x,\beta)=0\}$.
Determining the structure of an orthogonal complement over local ring is much easier than over global ring.
Fortunately, following lemma says that we may interchange the order of localization and orthogonal complementation.
\end{rmk}

\vskip 1em


\begin{lem} \label{p}
Let $L$ be a $\mathbb Z$-lattice and $K$ be a $\z$-sublattice of $L$.
Then $(K\otimes \mathbb Z_p)^{\perp}=K^{\perp}\otimes \mathbb Z_p$.
\end{lem}
\begin{proof}
Clearly, $(K\otimes \mathbb Z_p)^{\perp}\supseteq K^{\perp}\otimes \mathbb Z_p$ holds.
Let $v \in (K\otimes \mathbb Z_p)^{\perp} =\{v \in L \otimes \mathbb Z_p | B(v,K\otimes \mathbb Z_p)=0\}$.
Then $v$ can be written as $v=\alpha^{-1}(\alpha v)$ where $\alpha \in \mathbb Z $ with $p \nmid \alpha$ and $\alpha v \in L$.
Since $B(\alpha v,K)=0$, we have $v=\alpha^{-1}(\alpha v) \in K^{\perp}\otimes \mathbb Z_p$.
\end{proof}
\vskip 1em

\section{Local structure of orthogonal complement}

Throughout this section, we write
$$\mathbf a_2=(a_{i_1},\cdots, a_{i_{10}})$$
and assume that $\beta \in \z^{10}$ is a vector with
$$B_{\mathbf a_2}(\alpha,\beta)=1.$$

In this section, we study some criteria to figure out the local structure of $\Z_p$-sublattice $$L_{\mathbf a_2,\beta}\otimes \Z_p=(\z_p\alpha+\z_p\beta)^{\perp}$$ of $(\Z_p^{10},Q_{\mathbf a_2})$. 
The following lemma suggests an effective direction to discriminate whether $L_{\mathbf a_2, \beta}$ is locally even universal.

\begin{prp} \label{uni}
Suppose $p$ be an odd prime.

\begin{itemize}
\item[(1)]
If there are at least 6 units of $\Z_p$ (by admitting a recursion) in $\{a_{i_1}, \cdots, a_{i_{10}}\}$, then $L_{\mathbf a_2, \beta} \otimes \mathbb Z_p$ is (even) universal for $\beta \in \Z^{10}$ with $B_{\mathbf a_2}(\alpha, \beta)=1$.


\item[(2)]
If there are at least 5 units and at least 2 prime elements of $\Z_p$ (by admitting a recursion) in $\{a_{i_1}, \cdots, a_{i_{10}}\}$, then $L_{\mathbf a_2, \beta} \otimes \mathbb Z_p$ is (even) universal for $\beta \in \Z^{10}$ with $B_{\mathbf a_2}(\alpha, \beta)=1$.

\end{itemize}
\end{prp}
\begin{proof}
Recall that $(L_{\mathbf a_2, \beta},Q_{\mathbf a_2})$ is the orthogonal complement $(\mathbb Z \alpha + \mathbb Z \beta)^{\perp}$ in $(\mathbb Z^{10},Q_{\mathbf a_2})$. 

(1) If $\mathbb Z_p \alpha + \mathbb Z_p \beta$ is unimodular, then we may use 82:15, 91:9 in \cite{O} and Lemma \ref{p} to see that $L_{\mathbf a_2,\beta}\otimes \mathbb Z_p=(\mathbb Z_p \alpha + \mathbb Z_p \beta)^{\perp}$ contains a quaternary unimodular $\mathbb Z_p$-sublattice. This implies that $(L_{\mathbf a_2,\beta}\otimes \mathbb Z_p,Q_{\mathbf a_2})$ is (even) universal from 92:1b in \cite{O}.

Now assume that $\mathbb Z_p \alpha + \mathbb Z_p \beta$ is not unimodular.
Since its discriminant $$d(\Z_p\alpha+\Z_p\beta)=Q_{\mathbf a_2}(\alpha)Q_{\mathbf a_2}(\beta)-B_{\mathbf a_2}(\alpha,\beta)^2=Q_{\mathbf a_2}(\alpha)Q_{\mathbf a_2}(\beta)-1$$ is not unit of $\z_p$, we have that $Q_{\mathbf a_2}(\alpha) \in \Z_p^{\times}$.
Hence $\Z_p\alpha+\Z_p\beta$ is diagonalized as $\mathbb Z_p \alpha \perp \mathbb Z_p(\beta-Q_{\mathbf a_2}(\alpha)^{-1}\alpha)$, namely, $$\mathbb Z_p \alpha + \mathbb Z_p \beta=\mathbb Z_p \alpha \perp \mathbb Z_p(\beta-Q_{\mathbf a_2}(\alpha)^{-1}\alpha).$$
Since $p$ is odd prime, we may take an orthogonal $\Z_p$-basis $\mathbf v_1,\mathbf v_2,\cdots,\mathbf v_{10}$ of $(\Z_p^{10},Q_{\mathbf a_2})$ for which $$\begin{cases}\mathbf v_1=\alpha \\ Q_{\mathbf a_2}(\mathbf v_2),\cdots,Q_{\mathbf a_2}(\mathbf v_6) \in \Z_p^{\times}.\end{cases}$$
For $B_{\mathbf a_2}(\mathbf v_1, \beta-Q_{\mathbf a_2}(\alpha)^{-1}\alpha)=0$, we may write 
$$\beta-Q_{\mathbf a_2}(\alpha)^{-1}\alpha=p^{t_2}u_2\mathbf v_2+p^{t_3}u_3\mathbf v_3+\cdots+p^{t_{10}}u_{10}\mathbf v_{10}$$ for some $t_i \in \mathbb N_0$ and $u_i \in \Z_p^{\times}$.
Without loss of generality, let $0\le t_2\le t_3\le \cdots \le t_6$.
And put $M_{\beta}:=\Z_p \mathbf v_2+\cdots+\Z_p\mathbf v_6$, $\widetilde{\beta}:=\sum_{k=2}^6p^{t_k-t_2}u_k\mathbf v_k \in \mathbb Z_p^{10}$ and $\widetilde{\gamma}:=\mathbf v_2 \in \mathbb Z_p^{10}$.\\
If ${Q_{\mathbf a_2}}(\widetilde{\beta}) \in \mathbb Z_p^{\times}$, then $M_{\beta} \cap (\Z_p \widetilde{\beta})^{\perp}$ is a quaternary unimodular $\mathbb Z_p$-lattice by 82:15 and 91:9 in \cite{O}.
So, by Lemma \ref{p} and 92:1b in \cite{O}, $\Z_p$-sublattice $M_{\beta}\cap (\Z_p \widetilde{\beta})^{\perp}$ of $L_{\mathbf a_2,\beta}\otimes \Z_p$ is (even) universal, yielding that $L_{\mathbf a_2,\beta}\otimes \Z_p$ is also (even) universal.\\
If ${Q_{\mathbf a_2}}(\widetilde{\beta}) \in p\mathbb Z_p$, then $M_{\beta} \cap (\mathbb Z_p \widetilde{\beta}+\mathbb Z_p \widetilde{\gamma})^{\perp}$ is a ternary unimodular $\mathbb Z_p$-lattice by 82:15 and 91:9 in \cite{O}. 
So, by using Lemma \ref{p} and 92:1b in \cite{O}, we may see that the $\Z_p$-sublattice $M_{\beta} \cap (\mathbb Z_p \widetilde{\beta}+\mathbb Z_p \widetilde{\gamma})^{\perp}$ of $L_{\mathbf a_2,\beta}\otimes \Z_p$ is (even) universal, yielding that $L_{\mathbf a_2,\beta}\otimes \Z_p$ is also (even) universal.

(2) If $\mathbb Z_p \alpha + \mathbb Z_p \beta$ is unimodular, then $L_{\mathbf a_2, \beta}\otimes \mathbb Z_p$ contains a ternary unimodular $\mathbb Z_p$-sublattice by 82:15, 91:9 in \cite{O} and Lemma \ref{p}. 
This implies that $L_{\mathbf a_2, \beta}\otimes \mathbb Z_p$ is (even) universal by 92:1b in \cite{O}.

Assume that $\mathbb Z_p \alpha + \mathbb Z_p \beta$ is not unimodular.
Since $d(\Z_p\alpha+\Z_p\beta)=Q_{\mathbf a_2}(\alpha)Q_{\mathbf a_2}(\beta)-B_{\mathbf a_2}(\alpha,\beta)^2=Q_{\mathbf a_2}(\alpha)Q_{\mathbf a_2}(\beta)-1 \in p\Z_p$, $Q_{\mathbf a_2}(\alpha) \in \Z_p^{\times}$ and so $\Z_p\alpha+\Z_p\beta$ is diagonalized as $\mathbb Z_p \alpha \perp \mathbb Z_p(\beta-Q_{\mathbf a_2}(\alpha)^{-1}\alpha)$, i.e., $$\mathbb Z_p \alpha + \mathbb Z_p \beta=\mathbb Z_p \alpha \perp \mathbb Z_p(\beta-Q_{\mathbf a_2}(\alpha)^{-1}\alpha).$$Since $p$ is odd prime, there is an orthogonal $\Z_p$-basis $\mathbf v_1,\mathbf v_2,\cdots,\mathbf v_{10}$ of $(\Z_p^{10},Q_{\mathbf a})$ for which 
$$\begin{cases}\mathbf v_1=\alpha \\ 
Q_{\mathbf a_2}(\mathbf v_2),Q_{\mathbf a_2}(\mathbf v_3),Q_{\mathbf a_2}(\mathbf v_4),Q_{\mathbf a_2}(\mathbf v_5) \in \Z_p^{\times} \\
 Q_{\mathbf a_2}(\mathbf v_6),Q_{\mathbf a_2}(\mathbf v_7) \in p \Z_p^{\times}.\end{cases}$$
For $B_{\mathbf a_2}(\mathbf v_1, \beta-Q_{\mathbf a_2}(\alpha)^{-1}\alpha)=0$, we may write $$\beta-Q_{\mathbf a_2}(\alpha)^{-1}\alpha=p^{t_2}u_2\mathbf v_2+p^{t_3}u_3\mathbf v_3+\cdots+p^{t_{10}}u_{10}\mathbf v_{10}$$ where $t_i \in \mathbb N_0$ and $u_i \in \Z_p^{\times}$.
Without loss of generality, let $0\le t_2\le t_3\le \cdots \le t_5$ and $t_6 \le t_7$.
If $t_2 \le t_6$, then we may show that $L_{\mathbf a_2, \beta}\otimes \mathbb Z_p$ contains a binary unimodular $\Z_p$-sublattice and a binary $p$-modular $\Z_p$-sublattice through a similar way used in (1), which yields that $L_{\mathbf a_2,\beta}\otimes \Z_p$ is (even) universal.\\
If $t_2>t_6$, then $L_{\mathbf a_2,\beta}\otimes \mathbb Z_p$ contains a quaternary $\Z_p$-sublattice $$\sum_{k=2}^5\Z_p \mathbf w_k$$  where $\mathbf w_k=\mathbf v_k-p^{t_k-t_6}u_ku_6^{-1}Q_{\mathbf a_2}(\mathbf v_k)Q_{\mathbf a_2}(\mathbf v_6)^{-1}\mathbf v_6.$
Note that \begin{equation}\label{w} \text {$\begin{cases}Q_{\mathbf a_2}(\mathbf w_k) \equiv Q_{\mathbf a_2}(\mathbf v_k) \nequiv 0 \pmod{p}& \\ B_{\mathbf a_2}(\mathbf w_i,\mathbf w_j) \equiv 0 \pmod{p} & \text{for } i \neq j .\end{cases}$} \end{equation}
From the modulus condition (\ref{w}), we may obtain $d(\sum_{k=2}^5\Z_p\mathbf w_k) \in \Z_p^{\times}$, yielding that the quaternary $\Z_p$-sublattice $\sum_{k=2}^5\Z_p\mathbf w_k$ of $L_{\mathbf a_2,\beta}\otimes \Z_p$ is unimodular.
So, $L_{\mathbf a_2, \beta}\otimes \mathbb Z_p$ is (even) universal  by 92:1b in \cite{O}.
This completes the proof.
\end{proof}

\vskip 0.8em

\begin{rmk} \label{r1}
From now on throughout this paper, we always assume that \begin{equation}\label{f7}i_1=1, \ i_2=2, \cdots ,i_7=7. \end{equation}
Then for any prime $p \ge 17$, since $$\begin{cases}a_i < 2^i <17 \le p & \text{ for } i \le 5 \\ a_i<2^i <17^2 \le p^2 & \text{ for }i \le 7,\end{cases}$$ 
$L_{\mathbf a_2, \beta}\otimes \mathbb Z_p$ is (even) universal for any vector $\beta$ with $B_{\mathbf a_2}(\alpha, \beta)=1$ by Proposition \ref{uni}.
Hence when we examine whether there are $10$-subtuple $\mathbf a_2$ of $\mathbf a \in A$ and  $\beta \in \mathbb Z^{10}$ with $B_{\mathbf a_2}(\alpha,\beta)=1$ for which $L_{\mathbf a_2, \beta}$ is locally even universal under the condition (\ref{f7}), we only need to pay attention to the local structure $L_{\mathbf a_2, \beta}\otimes \z_p$ for primes $p$ less than or equal to 13, i.e., it is enough to check that whether $L_{\mathbf a_2, \beta}\otimes \mathbb Z_p$ are even universal for all primes $p\le 13$.

For each $\mathbf a \in A$, we define $P(\mathbf a)$ by the set of odd primes satisfying none of the following two conditions \\
\begin{equation} \label{u6}
\text{there are at least 6 units of $\z_p$ in }\{a_1, \cdots, a_7\} 
\end{equation}
and
\begin{equation} \label{u5p2}
\text{there are 5 units and 2 primes  of $\z_p$ in }\{a_1, \cdots, a_7\}
\end{equation}
by admitting a recursion.
Recall that for each $\mathbf a \in A$, we have $P(\mathbf a) \subseteq \{3,5,7,11,13\}$ from the assumption $i_1=1,i_2=2,\cdots,i_7=7$.
We define $$A(p):=\{\mathbf a\in A |p \in P(\mathbf a) \}$$ for odd prime $p$.
Then $A(p) \not=\emptyset$ only if $p \le 13$.

In order to do sort every $\mathbf a\in A$ for which $(L_{\mathbf a_2, \beta},Q_{\mathbf a_2})$ is locally even universal for some $10$-subtuple $\mathbf a_2=(a_{i_1},\cdots,a_{i_{10}})$ with (\ref{f7}) of $\mathbf a=(a_1,\cdots,a_{16})$ and $\beta=(\beta_1,\cdots,\beta_{10}) \in \mathbb Z^{10}$ with $B_{\mathbf a_2}(\alpha, \beta)=\sum _{k=1}^{10}a_{i_k}\cdot \beta_k=1$, it would suffice to go through the process to check the followings

\begin{itemize}
\item[(1)]
For $\mathbf a \in A(p) \setminus \bigcup\limits_{p'>p}A(p')$, whether there are $10$-subtuple $\mathbf a_2$ with (\ref{f7}) of $\mathbf a=(a_1,\cdots,a_{16})$ and $\beta \in \mathbb Z^{10}$ such that $(L_{\mathbf a_2, \beta}\otimes \mathbb Z_q,Q_{\mathbf a_2})$ are even universal for all primes $q\le p$ for each $p \in \{3,5,7,11,13 \}$.
\item[(2)]
For $\mathbf a \in A \setminus \bigcup\limits_{p=3}^{13}A(p)=:A(2)$, whether there are $10$-subtuple $\mathbf a_2$ with (\ref{f7}) of $\mathbf a=(a_1,\cdots,a_{16})$ and $\beta \in \mathbb Z^{10}$ such that $(L_{\mathbf a_2, \beta}\otimes \mathbb Z_2,Q_{\mathbf a_2})$ is even universal.
\end{itemize}
For $\mathbf a \in A$ except which fails to be confirmed (1) or (2) holds, we may claim that there are $10$-subtuple $\mathbf a_2$ of $\mathbf a \in A$ and  $\beta \in \mathbb Z^{10}$ with $B_{\mathbf a_2}(\alpha,\beta)=1$ for which $L_{\mathbf a_2, \beta}$ is locally even universal

\vskip 1em

\end{rmk}

The Lemma \ref{nd} is a key lemma which is mainly used to calculate the structure of the orthogonal complement $L_{\mathbf a_2,\beta}\otimes \Z_p=(\Z_p\alpha+\Z_p\beta)^{\perp}$ of $(\z_p^{10},Q_{\mathbf a_2})$ over non-dyadic local ring.
\begin{lem} \label{nd}
For an odd prime $p$, let $\Z_p\alpha+\Z_p\beta$ be unimodular $\z_p$-sublattice of $(\z_p^{10},Q_{\mathbf a_2})$ (i.e., $Q_{\mathbf a_2}(\alpha)Q_{\mathbf a_2}(\beta)-1 \in \z_p^{\times}$) and there are at least three units of $\Z_p$ in $\{a_{i_1},\cdots, a_{i_{10}}\}$ by admitting a recursion.
By rearranging the order, suppose that $a_{i_1},a_{i_2},a_{i_3} \in \Z_p^{\times}$.
Then $L_{\mathbf a_2,\beta} \otimes \Z_p=(\Z_p\alpha+\Z_p\beta)^{\perp}$ is isometric to $$\left<\frac{a_{i_1}a_{i_2}a_{i_3}}{Q_{\mathbf a_2}(\alpha)Q_{\mathbf a_2}(\beta)-1}\right>\perp\left<a_{i_4}\right>\perp \cdots \perp\left<a_{i_{10}}\right>.$$
\end{lem}
\begin{proof}
It follows from 82:15, 91:9, 92:1b, and 92:3 in \cite{O}.
\end{proof}

\vskip 0.8em

\begin{rmk}
Note that when $p\not=3$ is an odd prime, $a_{i_1},a_{i_2},$ and $a_{i_3}$ are all units of $\Z_p$. 
In virtue of Lemma \ref{nd}, we may easily figure out the exact structure of $L_{\mathbf a_2, \beta}\otimes \mathbb Z_p$ when $p$ is an odd prime and $\Z_p \alpha +\Z_p \beta$ is unimodular (we try to take a suitable vector $\beta$ for which $\Z_p\alpha+\Z_p\beta$ is unimodular for primes $p$ to review based on Remark \ref{r1}, i.e., for $ p \in P(\mathbf a)$).
On the other hand, when $p=2$, since the criterion in Lemma \ref{nd} does not hold, we need to find some even universal sublattice of $L_{\mathbf a_2, \beta}\otimes \mathbb Z_2$ directly to show that $L_{\mathbf a_2, \beta}\otimes \mathbb Z_2$ is even universal.
In the following lemmas, we introduce some criteria which are used to claim that $L_{\mathbf a_2, \beta}\otimes \mathbb Z_2$ is even universal by showing that $L_{\mathbf a_2,\beta}\otimes \Z_2$ contains an even universal $\Z_2$-sublattice.
\end{rmk}

\begin{lem} \label{odd3}
For $\mathbf a=(a_1,\cdots,a_{10}) \in \N^{10}$ and a vector $\beta=(\beta_1,\cdots,\beta_{10}) \in \Z ^{10}$ with $B_{\mathbf a}(\alpha,\beta)=1$, suppose that there are distinct indices $1\le j_1,j_2,j_3\le 10$ such that 
$$\begin{cases} \beta_{j_1}=\beta_{j_2}=\beta_{j_3} &  \\
a_{j_1}\equiv a_{j_2}\equiv a_{j_3} \equiv 1 \pmod{2}. & \end{cases}$$
\begin{itemize}
\item[(1)] If $a_{j_1}\nequiv a_{j_2} \pmod{4}$ or $a_{j_1}\nequiv a_{j_3} \pmod{4}$, then $L_{\mathbf a,\beta}\otimes \Z_2$ is even universal.
\item[(2)] If $a_{j_1}\equiv a_{j_2} \equiv a_{j_3} \pmod{4}$, then $L_{\mathbf a,\beta}\otimes \Z_2$ represents every $2$-adic integer with odd order in $\Z_2$.
\item[(3)] If $a_{j_1}\equiv a_{j_2} \equiv a_{j_3} \pmod{4}$ and there is an extra $j_4$ satisfying $\beta_{j_1}=\beta_{j_4}$ and $a_{j_4} \equiv 2 \pmod{4}$, then $L_{\mathbf a,\beta}\otimes \Z_2$ is even universal.
\end{itemize}
\end{lem}
\begin{proof}
(1) Note that $L_{\mathbf a,\beta} \otimes \Z_2$ contains a $\Z_2$-sublattice $$\Z_2( a_{j_2}\mathbf e_{j_1}-a_{j_1}\mathbf e_{j_2})+\Z_2( a_{j_3}\mathbf e_{j_2}-a_{j_2}\mathbf e_{j_3})$$ which is isometric to even universal hyperplane $\mathbb H\cong \begin{pmatrix}0&1\\ 1&0\end{pmatrix}$. This yields the claim.\\
(2) Since $L_{\mathbf a,\beta} \otimes \Z_2$ contains a $\Z_2$-sublattice 
$$\Z_2( a_{j_2}\mathbf e_{j_1}-a_{j_1}\mathbf e_{j_2})+\Z_2( a_{j_3}\mathbf e_{j_2}-a_{j_2}\mathbf e_{j_3})$$ which is isometric to $\mathbb A \cong \begin{pmatrix}2&1\\ 1&2\end{pmatrix}$ and $\mathbb A \otimes \z_2$ represents every prime element in $\Z_2$. 
This yields the claim. \\
(3) Note that $L_{\mathbf a,\beta} \otimes \Z_2$ contains a $\Z_2$-sublattice 
$$\Z_2( a_{j_2}\mathbf e_{j_1}-a_{j_1}\mathbf e_{j_2})+\Z_2( (a_{j_3}+a_{j_4})\mathbf e_{j_2}-a_{j_2}(\mathbf e_{j_3}+\mathbf e_{j4}))$$ which is isometric to even universal hyperplane $\mathbb H\cong \begin{pmatrix}0&1\\ 1&0\end{pmatrix}$.
This yields the claim.\\
\end{proof}

\vskip 0.8em

\begin{lem} \label{odd2}
For $\mathbf a=(a_1,\cdots,a_{10}) \in \N^{10}$ and a vector $\beta=(\beta_1,\cdots,\beta_{10}) \in \Z ^{10}$ with $B_{\mathbf a}(\alpha, \beta)=1$, suppose that 
$$\beta_{j_1}=\beta_{j_2}=\cdots =\beta_{j_l}$$ for all distinct $j_k$'s (the order $j_k<j_{k+1}$ is not required), then $L_{\mathbf a,\beta} \otimes
 \Z_p$ contains a $\Z_p$-sublattice which is isometric to 
 $$\left<b_1\right>\perp \cdots \perp\left<b_{l-1}\right>$$ where $b_h=a_{j_{h+1}}(\sum \limits_{k=1}^{h} a_{j_k})(\sum \limits_{k=1}^{h+1} a_{j_k})$ for any prime $p$.
\end{lem}
\begin{proof}
Put $\mathbf v_h:=a_{j_{h+1}}(\mathbf e_{j_1}+\mathbf e_{j_2}+\cdots +\mathbf e_{j_{h}})-(a_{j_1}+a_{j_2}+\cdots+a_{j_{h}})\mathbf e_{j_{h+1}}$ for $h=1,2,\cdots, l-1$.
Then $\mathbf v_1,\mathbf v_2,\cdots,\mathbf v_{l-1} \in L_{\mathbf a,\beta}$ and they are mutually orthogonal.
Therefore $L_{\mathbf a,\beta}\otimes \Z_p$ contains a $\z_p$-sublattice 
$$\displaystyle\bigperp_{j=1}^{l-1}\Z_p \mathbf v_j$$ with $Q_{\mathbf a}(\mathbf v_j)=b_j$.
This yields the claim.
\end{proof}

\vskip 0.8em

\begin{lem} \label{12}
For $\mathbf a=(a_1,\cdots,a_{10}) \in \N^{10}$ with $(a_1,a_2)=(1,2)$ and $\beta:=-\mathbf e_1+\mathbf e_2 \in \Z^{10}$, $B_{\mathbf a}(\alpha,\beta)=1$ holds.
\begin{itemize}
\item[(1)] If there are three odd integers in $\{a_3,\cdots, a_{10}\}$, then $L_{\mathbf a,\beta}\otimes \Z_2$ is even universal.
\item[(2)]
For all distint $l$ indices $3\le j_1,\cdots, j_l \le 10$ (the order $j_k<j_{k+1}$ is not required), $L_{\mathbf a,\beta} \otimes \Z_p$ contains a $\Z_p$-sublattice which is isometric to $\left<b_0\right>\perp\left<b_1\right>\perp \cdots \perp\left<b_{l-1}\right>$ where 
$$b_h=\begin{cases}\left( \frac{2a}{\gcd (4,a)} \right)^2+2\left( \frac{a}{\gcd (4,a)} \right)^2 +a \left( \frac{-4}{\gcd (4,a)} \right)^2 &\text{when } h=0 \\
a_{j_{h+1}}(\sum \limits_{k=1}^{h} a_{j_k})(\sum \limits_{k=1}^{h+1} a_{j_k}) & \text{when } h\ge 1\end{cases}$$
where $a:=\sum \limits_{k=1}^la_{j_k}$ for any prime $p$.
\end{itemize}
\end{lem}
\begin{proof}
(1) Without loss of generality, let $a_3,a_4,a_5$ be odd.\\
If $a_3 \nequiv a_4 \pmod{4}$ or $a_3 \nequiv a_5 \pmod{4}$, then by Lemma \ref{odd3} (1), it is done.\\
If $a_3\equiv a_4 \equiv a_5 \pmod{4}$, then $L_{\mathbf a, \beta}\otimes \Z_2$ contains a $\Z_2$-sublattice
$$\Z_2(2(a_3+a_4+a_5)\mathbf e_1+(a_3+a_4+a_5)\mathbf e_2-4(\mathbf e_3+\mathbf e_4+\mathbf e_5))\perp (\Z_2(a_4\mathbf e_3-a_3\mathbf e_4)+\Z_2(a_5\mathbf e_4-a_4\mathbf e_5))$$
which is isometric to even universal $\Z_2$-lattice $\left<6\right>\perp \mathbb A$.
 
(2)$L_{\mathbf a,\beta}\otimes \Z_p$ contains $\perp _{h=0}^{l-1}\Z_p \mathbf v_h$ where
$$\mathbf v_h =\begin{cases}\frac{2a}{\gcd (4,a)}\mathbf e_1+\frac{a}{\gcd (4,a)}\mathbf e_2-\frac{4}{\gcd (4,a)}(\mathbf e_{j_1}+\cdots + \mathbf e_{j_l}) &\text{for } h=0
\\a_{j_{h+1}}(\mathbf e_{j_1}+\mathbf e_{j_2}+\cdots +\mathbf e_{j_{h}})-(a_{j_1}+a_{j_2}+\cdots+a_{j_{h}})\mathbf e_{j_{h+1}} &\text{for }h\ge 1.\end{cases}$$
with $Q_{\mathbf a}(\mathbf v_h)=b_h$. 
This yields the claim.
\end{proof}

\vskip 0.8em

\begin{prp}\label{delta2}
A $\Z_2$-lattice which is isometric to $\left<b_1\right>\perp\left<b_2\right>\perp\left<b_3\right>\perp\left<b_4\right>$ with $(\ord_2(b_1), \ord_2(b_2), \ord_2(b_3), \ord_2( b_4))\in \Delta_2$ is even universal where 
\begin{align*}\Delta_2:=\{ & (1,1,1,1), (1,1,1,2),(1,1,1,3), (1,1,2,2), \\ &(1,1,2,3),(1,1,2,4),(1,2,2,3), (1,2,3,3),(1,2,3,4) \}.\end{align*}
\end{prp}
\begin{proof}
We leave it to reader.
\end{proof}

\vskip 0.8em

\begin{rmk} \label{r2}
If there are three odd integers by admitting a recursion among components of $\mathbf a=(a_1, \cdots, a_{16})$, then we try to take its suitable $10$-subtuple $\mathbf a_2$ and a vector $\beta \in \Z^{10}$ satisfying the assumption of Lemma \ref{odd3} to get a locally even universal $L_{\mathbf a_2,\beta}$ (of course by simultaneously considering its local structure at other prime spots $p \in P(\mathbf a)$). 
On the other hand, when there are no three odd coefficients among $(a_1, \cdots, a_{16})$ (note that there is at least one odd integer $a_1=1$) or it is difficult to take $\mathbf a_2$ and $\beta \in \Z^{10}$ satisfying the assumption of Lemma \ref{odd3} from any other issues caused over a non-dyadic prime spot $p \in P(\mathbf a)$, we may use Lemma \ref{odd2} to show the existence of an even universal $\Z_2$-sublattice of $L_{\mathbf a_2,\beta}\otimes \Z_2$, yielding $L_{\mathbf a_2,\beta}\otimes \Z_2$ is even universal.
In this case, we try to take a vector $\beta \in \Z^{10}$ such that $\beta_{j_1}=\beta_{j_2}=\cdots =\beta_{j_l}$ for all distinct $j_k$'s and $a_{j_1}$ is unit and $a_{j_2}$ is prime of $\Z_2$, etc.
Note that after we arrange the order of indices satisfying the followings
$$\begin{cases}\ord_2(a_{j_1})=0<\ord_2(a_{j_2})=1 \le \ord_2(a_{j_{3}}) \le \cdots \le \ord_2(a_{j_{l_1}}) \\ \text{either $\ord_2(a_{j_{l-1}}) \le \ord_2(a_{j_{l}})$ or $\ord_2(a_{j_{l}})=1$}, \end{cases}$$
we have that $\ord_2(b_{j_h})=\ord_2(a_{j_{h+1}})$.
Hence if \begin{equation}\label{2-equ}\text{$\ord_2(a_{j_1})=0$ and $(\ord_2(a_{j_2}),\ord_2(a_{j_3}), \ord_2(a_{j_4}), \ord_2(a_{j_5})) \in \Delta_2,$}\end{equation} then we may obtain that $L_{\mathbf a_2,\beta}\otimes \Z_2$ contains an even universal $\Z_2$-sublattice by Lemma \ref{odd2} and Proposition \ref{delta2}.
\end{rmk}

\vskip 1em

$\bold{\ I. \ A(13)}$ \\
We show that when $\mathbf a \in A(13)$, there are $10$-subtuple $\mathbf a_2=(a_{i_1},\cdots,a_{i_{10}})$ of $\mathbf a=(a_1,\cdots,a_{16})$ with (\ref{f7}) and $\beta=(\beta_1,\cdots,\beta_{10}) \in \mathbb Z^{10}$ with $B_{\mathbf a_2}(\alpha, \beta)=\sum _{k=1}^{10}a_{i_k}\cdot \beta_k=1$ for which $L_{\mathbf a_2,\beta}$ is locally even universal.
From the construction of $A(13)$,  we have that
\begin{align*}
A(13)= \{ &(1,2,3,6,13a_5',13a_6',13a_7',a_8,\cdots,a_{16})\\
&(1,2,3,7,13a_5',13a_6',13a_7',a_8,\cdots,a_{16}) \\
&(1,2,4,5,13a_5',13a_6',13a_7',a_8,\cdots,a_{16}) \\
&(1,2,4,6,13a_5',13a_6',13a_7',a_8,\cdots,a_{16}) \\
&(1,2,4,7,13a_5',13a_6',13a_7',a_8,\cdots,a_{16}) \\
&(1,2,4,8,13a_5',13a_6',13a_7',a_8,\cdots,a_{16}) | a_i',a_i \in \mathbb N \} \subset A.
\end{align*}
Based on Remark \ref{r1}, we examine for each $\mathbf a \in A(13)$, whethere there are $10$-subtuple $\mathbf a_2=(a_{i_1},\cdots,a_{i_{10}})$ of $\mathbf a$ with (\ref{f7}) and $\beta \in \mathbb Z^{10}$ for which $L_{\mathbf a_2, \beta}\otimes \mathbb Z_q$ are even universal for all $q \in \{2,3,5,7,11,13\}$ which would be equivalent with that $L_{\mathbf a_2,\beta}$ is locally even universal. 
For $\mathbf a=(a_1,\cdots,a_4,13a_5',13a_6',13a_7', a_8,\cdots,a_{16}) \in A(13)$, we have
 $$a_5'=1, \quad a_6'\in \{1,2\}, \quad a_7'\in \{1,2,3,4\}.$$
Hence by Proposition \ref{uni}, it is required to take care the local structure $L_{\mathbf a_2,\beta} \otimes \z_q$ only for $q \in \{2,3,13\}$ in choosing $\mathbf a_2$ and $\beta \in \mathbb Z^{10}$ for locally even universal $L_{\mathbf a_2, \beta}$.
Moreover, for $\mathbf a \in A(13)$ with $(a_3,a_4,a_7') \neq (3,6,3)$, either there are at least $6$ units (by admitting a recursion) of $\Z_3$ or there are $5$ units and $2$ primes (by admitting a recursion) of $\Z_3$ in $\{a_1,a_2,\cdots, a_7\}$.
By Proposition \ref{uni} for any $10$-subtuple $\mathbf a_2=(a_{i_1},\cdots,a_{i_{10}})$ with (\ref{f7}) of $\mathbf a \in A(13)$ with $(a_3,a_4,a_7')\neq (3,6,3)$ and $\beta \in \mathbb Z^{10}$, $L_{\mathbf a_2,\beta}\otimes \Z_3$ is (even) universal. 
Hence for $\mathbf a \in A(13)$ with $(a_3,a_4,a_7)\neq (3,6,39)$, we need to take care local structure $L_{\mathbf a_2,\beta} \otimes \z_q$ only for $q \in \{2,13\}$.

\begin{itemize}
\item[1.]
We first consider $\mathbf a \in A(13)$ with $(a_1,a_2,a_3,a_4)=(1,2,3,7)$.

For $\mathbf a_2:=(a_1,\cdots,a_{10})$, if $a_i \not\equiv 0 \pmod{13}$ for some $8 \le i \le 10$, then $L_{\mathbf a_2,\mathbf e_1}\otimes \z_{13}$ is (even) univesal by Proposition \ref{uni} (2).
If $a_i \equiv 0 \pmod{13}$ for all $ 8 \le i \le 10$, then since the discriminant of $(\Z\alpha+\Z\beta,Q_{\mathbf a_2})$ 
$$Q_{\mathbf a_2}(\alpha)Q_{\mathbf a_2}(\beta)-B_{\mathbf a_2}(\alpha, \beta)^2\equiv 0 \cdot 1 -1 \nequiv 0 \pmod{13}$$ 
is a unit of $\z_{13}$ for $\beta:=\mathbf e_1 \in \Z^{10}$, we have that $\Z_{13}\alpha+\Z_{13}\mathbf e_1$ is unimodular.
By using Lemma \ref{nd}, we have that $L_{\mathbf a_2,\mathbf e_1}\otimes \Z_{13}(=(\Z_{13}\alpha+\Z_{13}\mathbf e_1)^{\perp})$ is isometric to $$\left<\frac{1\cdot 2 \cdot 3}{Q_{\mathbf a_2}(\alpha)Q_{\mathbf a_2}(\mathbf e_1)-1}\right>\perp\left<7\right>\perp\left<13\right>\perp\left<13a_6'\right>\perp \cdots \perp\left<a_{i_{10}}\right>.$$
Since $L_{\mathbf a_2,\mathbf e_1}\otimes \Z_{13}$ has a binay unimodular $\Z_{13}$-sublattice (which is isometric to $\left<\frac{1\cdot 2 \cdot 3}{Q_{\mathbf a_2}(\alpha)Q_{\mathbf a_2}(\mathbf e_1)-1}\right>\perp\left<7\right>$) and a binary $13$-modular $\Z_{13}$-sublattice (which is isometric to $\left<13\right>\perp\left<13a_6'\right>$),
$L_{\mathbf a_2,\mathbf e_1}\otimes \Z_{13}$ is (even) universal.

By using Lemma \ref{odd3} with $a_{j_1}=a_3=3, a_{j_2}=a_4=7, a_{j_3}=a_5=13$ and $\beta_{j_1}=\beta_{j_2}=\beta_{j_3}=0$, we show that $L_{\mathbf a_2,\mathbf e_1}\otimes \Z_2$ is even universal for $\mathbf a_2=(a_1,\cdots, a_{10})$.

Consequently, we conclude that when $\mathbf a \in A(13)$ with $(a_1,a_2,a_3,a_4)=(1,2,3,7)$, $L_{\mathbf a_2,\mathbf e_1}$ is locally even universal for $\mathbf a_2=(a_1,\cdots, a_{10})$.

\vskip 1em

\item[2.]
We secondly consider $\mathbf a \in A(13)$ with $(a_1,a_2,a_3,a_4)=(1,2,4,5)$.

Through similar arguments with the case of $(a_1,a_2,a_3,a_4)=(1,2,3,7)$, we could show that $L_{\mathbf a_2,\mathbf e_1}\otimes \Z_{13}$ is (even) universal for $\mathbf a_2=(a_1,\cdots,a_{10})$.

If $a_6'=1$, then we may use Lemma \ref{odd3}(3) with $a_{j_1}=a_4=5, a_{j_2}=a_5=13, a_{j_3}=a_6=13,$ and $a_{j_4}=a_2=2$ to assert that $L_{\mathbf a_2,\mathbf e_1}\otimes \Z_2$ is even universal where $\mathbf a_2=(a_1,\cdots,a_{10})$.
If $a_6'=2$, then by using the Lemma \ref{odd2} with $a_{j_1}=a_4=5, a_{j_2}=a_2=2,a_{j_3}=a_6=26,a_{j_4}=a_3=4,a_{j_5}=a_5=13$,
we may show that $L_{\mathbf a_2,\mathbf e_1}\otimes \Z_2$ where $\mathbf a_2=(a_1,\cdots,a_{10})$ has an even universal $\Z_2$-sublattice which is isometric to  $$\left<6\right>\perp\left<2\right>\perp\left<20\right>\perp\left<40\right>.$$
Therefore we obtain that when $\mathbf a \in A(13)$ with $(a_1,a_2,a_3,a_4)=(1,2,4,5)$, $L_{\mathbf a_2,\mathbf e_1}\otimes \Z_2$ is even universal for $\mathbf a_2=(a_1,\cdots, a_{10})$.

Consequently, we conclude that when $\mathbf a \in A(13)$ with $(a_1,a_2,a_3,a_4)=(1,2,4,5)$, $L_{\mathbf a_2,\mathbf e_1}$ is locally even universal for $\mathbf a_2=(a_1,\cdots,a_{10})$.

\vskip 1em

\item[3.]
Similarly with the above 1 and 2, one may show that when $\mathbf a \in A(13)$ with 
$$(a_1,a_2,a_3,a_4) \in \{(1,2,4,6), (1,2,4,8)\}$$ or 
$$(a_1,a_2,a_3,a_4)=(1,2,3,6) \text{ and } a_7\neq 39,$$ $L_{\mathbf a_2,\mathbf e_1}$ is locally even universal for $\mathbf a_2=(a_1,\cdots,a_{10})$.

\vskip 1em

\item[4.]
We now consider $\mathbf a \in A(13)$ with $(a_1,a_2,a_3,a_4)=(1,2,4,7)$.

For $\mathbf a_2=(a_1,\cdots,a_{10})$ and $\beta=\mathbf e_1-5\mathbf e_3+\mathbf e_4+\mathbf e_5 \in \Z^{10}$ with $B_{\mathbf a_2}(\alpha,\beta)=1$, 
we obtain that $L_{\mathbf a_2,\beta} \otimes \z_{13}=(\z_{13}\alpha+\z_{13}\beta)^{\perp}$ is even universal similarly with the above.

By using Lemma \ref{odd3}(1) with $a_{j_1}=a_1=1, a_{j_2}=a_4=7,a_{j_3}=a_5=13$, we show that $L_{\mathbf a_2,\beta}\otimes \Z_2$ is even universal.

Consequently, we conclude that when $\mathbf a \in A(13)$ with $(a_1,a_2,a_3,a_4)=(1,2,4,7)$, $L_{\mathbf a_2,\beta}$ is locally even universal for $\mathbf a_2=(a_1,\cdots,a_{10})$ and $\beta=\mathbf e_1-5\mathbf e_3+\mathbf e_4+\mathbf e_5 \in \Z^{10}$.

\vskip 1em

\item[5.]
We finally consider $\mathbf a \in A(13)$ with $(a_1,a_2,a_3,a_4,a_7)=(1,2,3,6,39)$.

Through a similar argument with the above, we have that $L_{\mathbf a_2,\mathbf e_1}\otimes \Z_{13}$ is (even) universal for $\mathbf a_2=(a_1,\cdots,a_{10})$.

If there is a unit of $\z_3$ in $\{a_8,a_9,a_{10}\}$, then by using Proposition \ref{uni}(2), we see that $L_{\mathbf a_2,\mathbf e_1}\otimes \z_3$ is (even) universal for $\mathbf a_2=(a_1,\cdots, a_{10})$.
If $a_8,a_9,$ and $a_{10}$ are all multiples of $3$, then for $\mathbf a_2=(a_1,\cdots,a_{10})$, from $Q_{\mathbf a_2}(\alpha) \not\equiv 1 \pmod{3}$, we have that $\z_3\alpha+\z_3\mathbf e_1$ is unimodular sublattice of $(\z_3^{10},Q_{\mathbf a_2})$.
By using Lemma \ref{nd}, we may see that $L_{\mathbf a_2,\mathbf e_1} \otimes \z_3=(\z_3\alpha+\z_3\mathbf e_1)^{\perp}$ contains a binary unimodular $\z_3$-sublattice and a binary $3$-modular $\z_3$-sublattice, yielding $L_{\mathbf a_2,\mathbf e_1} \otimes \z_3$ is (even) universal.

For $\mathbf a_2=(a_1,\cdots,a_{10})$, by using Lemma \ref{odd3}(1) with $a_{j_1}=a_3=3, a_{j_2}=a_5=13, a_{j_3}=a_7=13\cdot 3$, we may see that $L_{\mathbf a_2,\mathbf e_1} \otimes \Z_2$ is even universal.

Consequently, we conclude that when $\mathbf a \in A(13)$ with $(a_1,a_2,a_3,a_4,a_7) =(1,2,3,6,39)$, $L_{\mathbf a_2,\mathbf e_1}$ is locally even universal for $\mathbf a_2=(a_1,\cdots,a_{10})$.
\end{itemize}

\vskip 1em

$\bold{\ II. \ A(11)}$\\
Note that when $\mathbf a =(a_1,\cdots,a_{16}) \in A(11)$, we have 
\begin{align*}
(a_1,a_2,a_3,a_4) \in \{ & (1,1,3,5),\ (1,1,3,6), \ (1,2,2,5), \ (1,2,2,6), \ (1,2,3,4) ,\\
&                                      (1,2,3,5), \ (1,2,3,6), \ (1,2,3,7), \ (1,2,4,4), \ (1,2,4,5),\\
&                                      (1,2,4,6),\ (1,2,4,7), \ (1,2,4,8)\}
\end{align*}
and 
$$a_5,\ a_6, \ a_7 \in 11 \N.$$
When
\begin{align*}
(a_1,a_2,a_3,a_4) \in \{ & (1,1,3,5),\ (1,1,3,6), \ (1,2,2,5), \ (1,2,2,6), \ (1,2,3,4) ,\\
&                                      (1,2,3,5), \  (1,2,3,7), \ (1,2,4,4),\ (1,2,4,6), \ (1,2,4,8) \},
\end{align*}
one may show that $L_{\mathbf a_2, \mathbf e_1}$ is locally universal for the $10$-subtuple $\mathbf a_2=(a_1,\cdots,a_{10})$ of $\mathbf a$ through similar processings with $\bold{\ I. \  A(13)}$.
When 
$$(a_1,a_2,a_3,a_4) \in \{(1,2,3,6), \ (1,2,4,5), \ (1,2,4,7)\},$$
one may show that $L_{\mathbf a_2, -\mathbf e_1+\mathbf e_2}$ is locally universal for $10$-subtuple $\mathbf a_2=(a_1,\cdots,a_{10})$ of $\mathbf a$ through similar processings with $\bold{\ I. \  A(13)}$.
We omit the lengthy processings in this paper.

\vskip 1em

$\bold{\ III. \ A(7)}$\\
For $\mathbf a \in A(7)$, $\mathbf a$ satisfies one of the following three conditions
\begin{itemize}
\item[(1)]
There are five units and one prime element of $\Z_7$ in $\{a_1,a_2,\cdots, a_6\}$ by admitting a recursion and $a_7=49$.
\item[(2)]
There are four units of $\Z_7$ in $\{a_1,a_2,\cdots, a_7\}$ by admitting a recursion.
\item[(3)]
There are three units of $\Z_7$ in $\{a_1,a_2,\cdots, a_7\}$ by admitting a recursion. 
\end{itemize}

First we consider $\mathbf a \in A(7)$ satisfying (3), note that which is one of the forms
$$(1,2,3,7a_4',\cdots,7a_7',a_8,\cdots, a_{16}) \text{ or } (1,2,4,7a_4',\cdots,7a_7',a_8,\cdots, a_{16}) \in A.$$
Recall that based on Remark \ref{r1}, we examine for each $\mathbf a \in A(7) \setminus A(13) \cup A(11)$, whether there are $10$-subtuple $\mathbf a_2=(a_{i_1},\cdots,a_{i_{10}})$ of $\mathbf a$ with (\ref{f7}) and $\beta \in \mathbb Z^{10}$ for which $L_{\mathbf a_2, \beta}\otimes \mathbb Z_q$ are even universal for all $q \in \{ 2,3,5,7 \}$, yielding the $L_{\mathbf a_2,\beta}$ is locally even universal. 
From the construction of $A$, we have that $a_4'=1, a_5'\in \{1,2\},$ and $a_6'\in \{1,2,3,4\}$.
For $\mathbf a \in A(7)$ satisfying (3), we may use Proposition \ref{uni} to see that for any $10$-subtuple $\mathbf a_2$ of $\mathbf a$ with (\ref{f7}) and $\beta \in \Z^{10}$, $L_{\mathbf a_2,\beta} \otimes \Z_5$ is (even) universal.
Moreover, for $\mathbf a \in A(7)$ satisfying (3) if $(a_3,a_6',a_7') \not\in \{(3,3,3),(3,3,6) \}$, then either there are at least six units (by admitting a recursion) of $\Z_3$ or there are five units and two primes (by admitting a recursion) of $\Z_3$ in $\{a_1,a_2,\cdots, a_7\}$.
Therefore for any $10$-subtuple $\mathbf a_2$ with (\ref{f7}) of $\mathbf a$ with $(a_3,a_6,a_7) \not\in\{(3,21,21),(3,21,42)\}$ and $\beta \in \Z^{10}$, $L_{\mathbf a_2,\beta}\otimes \Z_3$ is (even) universal. 
Hence for cases of $(a_3,a_6,a_7) \not\in\{(3,21,21),(3,21,42)\}$, we only need to take care the local structure over $\z_p$ for $p=2$, $7$.
\begin{itemize}
\item[1.] First we consider $\mathbf a \in A(7)$ satisfying (3) with $a_3=3$ and $(a_6,a_7) \not\in\{(21,21),(21,42)\}$.
From the construction of $A(7)$, we have that
$$(a_4,a_5,a_6) \in \{(7,7,7),(7,7,14),(7,7,21),(7,14,14),(7,14,21),(7,14,28)\}.$$
If $(a_4,a_5,a_6)=(7,7,7)$, then for $\mathbf a_2=(a_1,\cdots,a_{10})$ and $\beta:=-\mathbf e_1-\mathbf e_2-\mathbf e_3-\mathbf e_4-\mathbf e_5+3\mathbf e_6 \in \Z^{10}$, we may use Lemma \ref{odd2} (with $a_{j_1}=a_1=1,a_{j_2}=a_2=2, a_{j_3}=a_3=3, a_{j_4}=a_4=7,a_{j_5}=a_5=7$) to show that $L_{\mathbf a_2,\beta}\otimes \Z_7$ contains an (even) universal $\Z_7$-sublattice which is isometric to
$$\left<2\cdot1\cdot3\right>\perp\left<3\cdot3\cdot6\right>\perp\left<7\cdot6\cdot13\right>\perp\left<7\cdot13\cdot20\right>,$$ which implies that $L_{\mathbf a_2,\beta} \otimes \Z_7$ is (even) universal.
And we may use Lemma \ref{odd3} (1) (with $a_{j_1}=1,a_{j_2}=a_3=3,a_{j_3}=a_4=7$) to yield that $L_{\mathbf a_2,\beta} \otimes \Z_2$ is even universal.
Consequently, we conclude that when $\mathbf a \in A(7)$ with $(a_1,\cdots,a_6)=(1,2,3,7,7,7)$, $L_{\mathbf a_2,\beta}$ is locally even universal for $\mathbf a_2=(a_1,\cdots,a_{10})$ and $\beta=-\mathbf e_1-\mathbf e_2-\mathbf e_3-\mathbf e_4-\mathbf e_5+3\mathbf e_6 \in \Z^{10}$.  

For the remaining $\mathbf a \in A(7)$ satisfying (3) with $a_3=3$ and $(a_6,a_7) \not\in\{(21,21),(21,42)\}$,
one may see that $L_{\mathbf a_2,\beta}$ is locally even universal for $\mathbf a_2=(a_1,\cdots,a_{10})$ and below $\beta \in \Z^{10}$ through a similar processing with the above.
We leave it to the reader.   
\begin{itemize}
\item[(1)]
$(a_4,a_5,a_6)=(7,7,7) \ ; \ \beta=-\mathbf e_1-\mathbf e_2-\mathbf e_3-\mathbf e_4-\mathbf e_5+3\mathbf e_6$.
\item[(2)]
$(a_4,a_5,a_6)=(7,7,14) \ ;  \ \beta=-\mathbf e_1-\mathbf e_2-\mathbf e_3-\mathbf e_4+4\mathbf e_5-\mathbf e_6$.
\item[(3)]
$(a_4,a_5,a_6)=(7,7,21) \ ; \ \beta=-\mathbf e_1-\mathbf e_2-\mathbf e_3-\mathbf e_4-\mathbf e_5+\mathbf e_6$.
\item[(4)]
$(a_4,a_5,a_6)=(7,14,14) \ ; \ \beta=-\mathbf e_1-\mathbf e_2-\mathbf e_3-\mathbf e_4-\mathbf e_5+2\mathbf e_6$.
\item[(5)]
$(a_4,a_5,a_6)=(7,14,21) \ ;  \ \beta=-\mathbf e_1-\mathbf e_2-\mathbf e_3+6\mathbf e_4-\mathbf e_5-\mathbf e_6$.
\item[(6)]
$(a_4,a_5,a_6)=(7,14,28) \ ; \ \beta=-\mathbf e_1-\mathbf e_2-\mathbf e_3-\mathbf e_4-\mathbf e_5+\mathbf e_6$.
\end{itemize}

\item[2.]Consider $\mathbf a \in A(7)$ satisfying (3) and $(a_3, a_6,a_7) \in \{ (3,21,21), (3,21,42)\}$.

One may see that $L_{\mathbf a_2,\beta} \otimes \Z_7$ is (even) universal for $\mathbf a_2=(a_1,\cdots,a_{10})$ and $\beta=-\mathbf e_1-\mathbf e_2-\mathbf e_3+6\mathbf e_4-\mathbf e_5-\mathbf e_6 \in \Z^{10}$ through a similar processing with the above.

If $a_i \nequiv 0 \pmod{3}$ for some $8 \le i \le 10$, then by Proposition \ref{uni}, $L_{\mathbf a_2,\beta}\otimes \Z_{3}$ is (even) universal for $\mathbf a_2=(a_1,\cdots,a_{10})$ and $\beta=-\mathbf e_1-\mathbf e_2-\mathbf e_3+6\mathbf e_4-\mathbf e_5-\mathbf e_6 \in \Z^{10}$.
If $a_i \equiv 0 \pmod{3}$ for all $8 \le i \le 10$, then since the discriminant of $(\z_3 \alpha+\z_3\beta,Q_{\mathbf a_2})$ is a unit of $\z_3$ for $\mathbf a_2=(a_1,\cdots,a_{10})$ and $\beta=-\mathbf e_1-\mathbf e_2-\mathbf e_3+6\mathbf e_4-\mathbf e_5-\mathbf e_6 \in \Z^{10}$, $(\z_3 \alpha+\z_3\beta,Q_{\mathbf a_2})$ is a unimodular sublattice of $(\z_3^{10},Q_{\mathbf a_2})$.
By using Lemma \ref{nd}, we have that $L_{\mathbf a_2,\beta} \otimes \Z_3(=(\Z_3\alpha+\Z_3\beta)^{\perp})$ is (even) universal for $\mathbf a_2=(a_1,\cdots,a_{10})$ and $\beta=-\mathbf e_1-\mathbf e_2-\mathbf e_3+6\mathbf e_4-\mathbf e_5-\mathbf e_6 \in \Z^{10}$.

Lastly, we may use Lemma \ref{odd3} (1) (with $a_{j_1}=a_1=1,a_{j_2}=a_3=3,a_{j_3}=a_6=7\cdot 3$) to show that $L_{\mathbf a_2,\beta} \otimes \Z_2$ is even universal for $\mathbf a_2=(a_1,\cdots,a_{10})$ and $\beta=-\mathbf e_1-\mathbf e_2-\mathbf e_3+6\mathbf e_4-\mathbf e_5-\mathbf e_6 \in \Z^{10}$.

Consequently, when $\mathbf a \in A(7)$ satisfies (3)  and $(a_3, a_6,a_7) \in \{ (3,21,21),$  $(3,21,42)\}$,  $L_{\mathbf a_2,\beta}$  is locally even universal for $\mathbf a_2=(a_1,\cdots,a_{10})$ and $\beta=-\mathbf e_1-\mathbf e_2-\mathbf e_3+6\mathbf e_4-\mathbf e_5-\mathbf e_6 \in \Z^{10}$.

\item[3.]Next consider $\mathbf a \in A(7)$ satisfying (3), $a_3=4$, and $a_i \not\equiv 0 \pmod{7}$ for some $8 \le i \le 16$.

By Proposition \ref{uni}, $L_{\mathbf a_2,\beta}\otimes \Z_q$ is (even) universal for any $10$-subtuple $\mathbf a_2$ with (\ref{f7}) of $\mathbf a$ and a vector $\beta \in \z^{10}$ with $B_{Q_{\mathbf a_2}}(\alpha, \beta)=1$ for every odd prime $q \not=7$.

We take $10$-subtuple $\mathbf a_2=(a_1,\cdots,a_7,a_{i_8},a_{i_9},a_{i_{10}})$ of $\mathbf a$ for which one of $a_{i_8},a_{i_9},$ and $a_{i_{10}}$ as an unit of $\Z_{7}$.
Since for $\beta^{(1)}:=\mathbf e_1$ and $\beta^{(2)}:=-\mathbf e_1+\mathbf e_2$ in $\Z^{10}$, one of the discriminant of $(\z_7\alpha+\z_7\beta^{(1)},Q_{\mathbf a_2})$  and $(\z_7\alpha+\z_7\beta^{(2)},Q_{\mathbf a_2})$ is an unit of $\z_7$, $(\z_7\alpha+\z_7\beta^{(1)},Q_{\mathbf a_2})$ or $(\z_7\alpha+\z_7\beta^{(2)},Q_{\mathbf a_2})$ is unimodular.
By using Lemma \ref{nd}, we may show that $L_{\mathbf a_2,\beta^{(1)}}\otimes \z_7$ or $L_{\mathbf a_2,\beta^{(2)}}\otimes \z_7$ is (even) universal.

For $\mathbf a \in A(7)$ satisfying (3) with $a_3=4$, by using Lemma \ref{odd3} and Lemma \ref{odd2}, one may show that both $L_{\mathbf a_2,\beta^{(1)}}\otimes \z_2$ and $L_{\mathbf a_2,\beta^{(2)}}\otimes \z_2$ are even universal for any $10$-subtuple of $\mathbf a_2$ with (\ref{f7}) of $\mathbf a$.

Consequently, we may conclude that when $\mathbf a \in A(7)$ satisfies (3), $a_3=4$, and $a_i \not\equiv 0 \pmod{7}$ for some $8 \le i \le 16$,  $L_{\mathbf a_2,\beta}$ is locally even universal for some $10$-subtuple $\mathbf a_2=(a_1,\cdots,a_7,a_{i_8},a_{i_9},a_{i_{10}})$ of $\mathbf a$ with $a_{i_k} \in \z_7^{\times}$ for some $8 \le k \le 10$  and $\beta \in \{\mathbf e_1,-\mathbf e_1+\mathbf e_2 \}$.

\item[4.]Let $\mathbf a \in A(7)$ satisfy (3), $a_3=4$, and $a_i \equiv 0 \pmod{7}$ for all $8 \le i \le 16$.

For such $\mathbf a$, there is something serious issue which is $L_{\mathbf a_2,\beta}$ is never locally even universal for any $10$-subtuple $\mathbf a_2$ of $\mathbf a$ and $\beta \in \Z^{10}$ with $B_{\mathbf a_2}(\alpha, \beta)=1$.
More precisely, $L_{\mathbf a_2,\beta} \otimes \Z_7$ is never (even) universal for any $10$-subtuple $\mathbf a_2$ of $\mathbf a$ and $\beta \in \Z^{10}$ with $B_{\mathbf a_2}(\alpha, \beta)=1$. 
We postpone to treat these candidates to Section 5.

\item[5.] Let $\mathbf a \in A(7)$ satisfy (1).\\
Then we have that 
$$(a_1,a_2,a_3,a_4) \in \{(1,2,3,7),(1,2,4,6),(1,2,4,7),(1,2,4,8)\}.$$
When $(a_1,a_2,a_3,a_4)=(1,2,3,7)$, we have that $(a_5,a_6) \in \{(11,24), (11,25), $ $(12,23),(12,24),(12,25),(12,26),(13,22),(13,23),(13,24),(13,25),(13,26), $ $(13,27)\}$. 
One may show that $L_{\mathbf a_2,-\mathbf e_1+\mathbf e_2}$ is locally even universal for $10$-subtuple $\mathbf a_2=(a_1,\cdots ,a_{10})$ of $\mathbf a \in A(7)$ satisfying (1) with $(a_1,a_2,a_3,a_4)=(1,2,3,7)$ case by case.\\
When $(a_1,a_2,a_3,a_4)=(1,2,3,6)$, we have that $(a_5,a_6) \in \{ (14,22),(14,23), $ $(14,24),(14,25),(14,26),(14,27) \}$.
One may show that $L_{\mathbf a_2,-\mathbf e_1+\mathbf e_2}$ (resp, $L_{\mathbf a_2,26\mathbf e_1-26\mathbf e_2+\mathbf e_6}$) is locally even universal for $10$-subtuple $\mathbf a_2=(a_1,\cdots ,a_{10})$ of $\mathbf a \in A(7)$ satisfying (1) with $(a_1,a_2,a_3,a_4)=(1,2,3,6)$ and $(a_5,a_6) \not=(14,27)$ (resp, $(a_5,a_6) =(14,27)$) case by case.\\
When $(a_1,a_2,a_3,a_4) \in \{ (1,2,4,7), (1,2,4,8) \}$, $L_{\mathbf a_2,\mathbf e_1}$ or $L_{\mathbf a_2,-\mathbf e_1+\mathbf e_2}$ is locally even universal for $10$-subtuple $\mathbf a_2=(a_1,\cdots ,a_{10})$ of $\mathbf a \in A(7)$ satisfying (1).\\

\item[6.] Let $\mathbf a \in A(7)$ satisfy (2).\\
Then we have that
\begin{align*}
(a_1,a_2,a_3,a_4) \in \{ & (1,1,1,3),(1,1,1,4),(1,1,2,2),(1,1,2,3), (1,1,2,4),\\
& (1,1,2,5),(1,1,3,3),(1,1,3,4),(1,1,3,5),(1,1,3,6),\\
&(1,2,2,2),(1,2,2,3),(1,2,2,4),(1,2,2,5),(1,2,2,6),\\
&(1,2,3,3),(1,2,3,4),(1,2,3,5),(1,2,3,6),(1,2,3,7),\\
&(1,2,4,4),(1,2,4,5),(1,2,4,6),(1,2,4,7),(1,2,4,8)\}.
\end{align*}
 When $(a_1,a_2,a_3,a_4) \in \{(1,1,1,3),(1,1,1,4),(1,1,2,2),(1,1,2,3),(1,1,2,5),$ $(1,1,3,4),(1,1,3,5),(1,1,3,6),(1,2,2,2),(1,2,2,4),(1,2,2,5),(1,2,2,6), (1,2,3,6),$  $(1,2,4,4),(1,2,4,6)\}$, $L_{\mathbf a_2,\mathbf e_1}$ is locally even universal for $10$-subtuple $\mathbf a_2=(a_1,\cdots ,a_{10})$ of $\mathbf a$.\\
When $(a_1,a_2,a_3,a_4) \in \{(1,2,2,3),(1,2,3,3), (1,2,3,4), (1,2,3,5), (1,2,4,8)\}$, $L_{\mathbf a_2,-\mathbf e_1+\mathbf e_2}$ is locally even universal for $10$-subtuple $\mathbf a_2=(a_1,\cdots ,a_{10})$ of $\mathbf a$.\\
When $(a_1,a_2,a_3,a_4)=(1,1,2,4)$, $L_{\mathbf a_2,-\mathbf e_1+\mathbf e_3}$ is locally even universal for $10$-subtuple $\mathbf a_2=(a_1,\cdots ,a_{10})$ of $\mathbf a$.\\
When $(a_1,a_2,a_3,a_4)=(1,2,4,5)$, $L_{\mathbf a_2,-\mathbf e_3+\mathbf e_4}$ is locally even universal for $10$-subtuple $\mathbf a_2=(a_1,\cdots ,a_{10})$ of $\mathbf a$.\\
When $(a_1,a_2,a_3,a_4)=(1,1,3,3)$, $L_{\mathbf a_2,-\mathbf e_3-\mathbf e_4-\mathbf e_5+\frac{14}{a_6}\mathbf e_6}$ is locally even universal for $10$-subtuple $\mathbf a_2=(a_1,\cdots ,a_{10})$ of $\mathbf a$.\\
When $(a_1,a_2,a_3,a_4)=(1,2,3,7)$ and $a_5\not=14$ (resp, $a_5=14$), $L_{\mathbf a_2,\beta}$ (resp, $L_{\mathbf a_2,-\mathbf e_1-\mathbf e_2-\mathbf e_3-\mathbf e_4+\mathbf e_5}$) is locally even universal for some $\beta \in \{\mathbf e_1, -\mathbf e_1+\mathbf e_2, -\mathbf e_2+\mathbf e_3, 4\mathbf e_2-\mathbf e_4, -2\mathbf e_1+\mathbf e_3\}$ and $10$-subtuple $\mathbf a_2=(a_1,\cdots ,a_{10})$ of $\mathbf a$.\\
When $(a_1,a_2,a_3,a_4)=(1,2,4,7)$,  $L_{\mathbf a_2,\mathbf e_1},L_{\mathbf a_2,-\mathbf e_1+\mathbf e_2},$ or $L_{\mathbf a_2,2\mathbf e_3-\mathbf e_4}$ is locally even universal for $10$-subtuple $\mathbf a_2=(a_1,\cdots ,a_{10})$ of $\mathbf a$.\\

\end{itemize}

\vskip 1em

$\bold{\ IV. \ A(5)}$ \\
For $\mathbf a \in A(5)$, $\mathbf a$ satisfies one of the following three conditions
\begin{itemize}
\item[(1)]
There are five units and at most one prime of $\Z_5$ in $\{a_1,a_2,\cdots, a_7\}$ by admitting a recursion.
\item[(2)]
There are four units of $\Z_5$ in $\{a_1,a_2,\cdots, a_7\}$ by admitting a recursion.
\item[(3)]
There are three units of $\Z_5$ in $\{a_1,a_2,\cdots, a_7\}$ by admitting a recursion. 
\end{itemize}
Based on Remark \ref{r1}, we examine for each $\mathbf a \in A(5) \setminus A(13)\cup A(11)\cup A(7)$, whether there are $10$-subtuple $\mathbf a_2$ of $\mathbf a$ with (\ref{f7}) and $\beta \in \mathbb Z^{10}$ with $B_{\mathbf a_2}(\alpha, \beta)=1$ for which $L_{\mathbf a_2, \beta}\otimes \mathbb Z_q$ are even universal for all $q \in \{ 2,3,5 \}$ yielding the $L_{\mathbf a_2,\beta}$ is locally even universal. 

Through similar processing with the above by using Proposition \ref{uni}, Lemma \ref{nd}, Lemma \ref{odd3}, Lemma \ref{odd2} and Lemma \ref{12}, one may see that $L_{\mathbf a_2,\beta}$ is locally even universal for some $10$-subtuple $\mathbf a_2$ of $\mathbf a$ with (\ref{f7}) and $\beta \in \mathbb Z^{10}$ with $B_{\mathbf a_2}(\alpha, \beta)=1$ for each $\mathbf a \in A(5)$ case by case except of the forms
\begin{equation} \label{5-1}(1,1,3,6,10,15a_6',15a_7',3a_8'\cdots, 3a_{16}') \in A(5) \text{ satisfying }(2),  \end{equation}
\begin{equation} \label{5-2} (1,1,3,5a_4',\cdots, 5a_{16}') \text{ and } (1,2,2,5a_4',\cdots, 5a_{16}') \in A(5) \text{ satisfying }(3).  \end{equation}
For the above $\mathbf a \in A(5)$, there is something serious issue which is $L_{\mathbf a_2,\beta}$ is never locally even universal for any $10$-subtuple $\mathbf a_2$ of $\mathbf a$ and $\beta \in \Z^{10}$ with $B_{\mathbf a_2}(\alpha, \beta)=1$.
For the candidates $\mathbf a \in A(5)$ in (\ref{5-1}), $L_{\mathbf a_2,\beta}\otimes \Z_3$ is never (even) universal and for the candidates $\mathbf a \in A(5)$ in  (\ref{5-2}), $L_{\mathbf a_2,\beta}\otimes \Z_5$ is never (even) universal.
We postpone to treat these candidates to Section 5.

\vskip 1em

$\bold{\ V. \ A(3)}$\\
When $\mathbf a \in A(3)$, $\mathbf a$ satisfies one of the following four conditions
\begin{itemize}
\item[(1)]
There are five units and at most one prime of $\Z_3$ in $\{a_1,a_2,\cdots, a_7\}$ by admitting a recursion.
\item[(2)]
There are four units of $\Z_3$ in $\{a_1,a_2,\cdots, a_7\}$ by admitting a recursion.
\item[(3)]
There are three units of $\Z_3$ in $\{a_1,a_2,\cdots, a_7\}$ by admitting a recursion. 
\item[(4)]
There are two units of $\Z_3$ in $\{a_1,a_2,\cdots, a_7\}$ by admitting a recursion. 
\end{itemize}
Based on Remark \ref{r1}, we examine for each $\mathbf a \in A(3) \setminus A(13)\cup A(11)\cup A(7)\cup A(5)$, whether there are $10$-subtuple $\mathbf a_2$ of $\mathbf a$ with (\ref{f7}) and and $\beta \in \mathbb Z^{10}$ for which $L_{\mathbf a_2, \beta}\otimes \mathbb Z_q$ are even universal for all $q \in \{ 2,3 \}$, yielding the $L_{\mathbf a_2,\beta}$ is locally even universal. 

Going through a similar process with the above by using Proposition \ref{uni}, Lemma \ref{nd}, Lemma \ref{odd3}, Lemma \ref{odd2} and Lemma \ref{12} with case by case consideration, one may check that there are $10$-subtuple $\mathbf a_2$ of $\mathbf a$ with (\ref{f7}) and $\beta \in \mathbb Z^{10}$ for which $L_{\mathbf a_2,\beta}$ is locally even universal for each $\mathbf a \in A(3)$ except of the forms
\begin{equation} \label{3-1}(1,2,3,6,8,8a_6',24,8a_8',\cdots, 8a_{16}') \in A(3) \text{ satisfying }(2),  \end{equation}
\begin{equation} \label{3-2} (1,1,a_3,\cdots,a_{16})\in A(3) \text{ satisfying }(3) \text{ or } (4)  \end{equation}
for which there is only one unit of $\z_3$ in $\{a_3,\cdots,a_{16}\}$ by admitting a recursion and $a_3+\cdots+a_{16} \equiv 1 \pmod{3}$, or
\begin{equation} \label{3-3}(1,1,3a_3',\cdots, 3a_{16}') \text{ and } (1,2,3a_3',\cdots, 3a_{16}')\in A(3) \text{ satisfying } (4). \end{equation}
We postpone to treat these exceptional candidates $\mathbf a$ to Section 5 where it is turned out that $L_{\mathbf a_2,\beta}$ is never locally even universal for any $10$-subtuple $\mathbf a_2$ of $\mathbf a$ and $\beta \in \Z^{10}$ with $B_{\mathbf a_2}(\alpha, \beta)=1$.
We might add in advance that for $\mathbf a \in A(3)$ of the form in (\ref{3-1}), $L_{\mathbf a_2,\beta}\otimes \Z_2$ is never even universal and for $\mathbf a \in A(3)$ of the form in (\ref{3-2}) or (\ref{3-3}), $L_{\mathbf a_2,\beta}\otimes \Z_3$ is never (even) universal.

\vskip 1em

$\bold{\ VI. \ A(2)}$ \\
Let $$A(2):=A\setminus \bigcup \limits_{p \in S}A(p) \text{ where } S=\{3,5,7,11,13\}.$$
Based on Remark \ref{r1}, we examine for each $\mathbf a \in A(2)$, whether there are $10$-subtuple $\mathbf a_2$ of $\mathbf a$ with (\ref{f7}) and and $\beta \in \mathbb Z^{10}$ for which $L_{\mathbf a_2, \beta}\otimes \mathbb Z_2$ is even universal, yielding the $L_{\mathbf a_2,\beta}$ is locally even universal. 

Going through a similar processing with the above by using Lemma \ref{odd3}, Lemma \ref{odd2} and Lemma \ref{12}, 
one may see that $L_{\mathbf a_2,\beta}$ is locally even universal for some $10$-subtuple $\mathbf a_2$ of $\mathbf a$ with (\ref{f7}) and $\beta \in \mathbb Z^{10}$ with $B_{\mathbf a_2}(\alpha, \beta)=1$ for each $\mathbf a \in A(2)$ case by case except  the candidates in $A'(2)$ of Table \ref{fd}.

\vskip 1em
\begin{rmk}
To finish this lengthy calculations, some computer programs were used.
In this chapter, we showed that for each $\mathbf a \in A \setminus A'$ except some dropouts in $A':=A'(2)\cup A'(3)\cup A'(5)\cup A'(7)$ with $A'(p)\subset A(p)$ in Table \ref{fd}, there is a constant $C_{\mathbf a}>0$ for which the universality of any $m$-gonal form having its first $16$ coefficients as $\mathbf a $ is characterized by the representability of every positive integer up to $C_{\mathbf a}(m-2)$.
The remaining dropouts $\mathbf a \in A'$ in Table \ref{fd} are treated in Section 5 and Section 6.
\end{rmk}

\begin{center}
\begin{table} 
\caption{Dropouts} \label{fd}
\label{t3}
    \begin{tabular}{ | c | l |}
     \hline
   \rule[-2.4mm]{0mm}{7mm}  $A'(7)$       &  $(1,2,4,7a_4',\cdots, 7a_{16}')$   \\
    \hline
\rule[-2.4mm]{0mm}{7mm}  $A'(5)$      
 & $(1,1,3,3,6,15,25,3a_8',\cdots,3a_{16}'),$\\
 & $(1,1,3,3,10,15,15a_7',3a_8',\cdots,3a_{16}'),$\\
& $(1,1,3,6,6,15,25,3a_8',\cdots,3a_{16}'),$\\
& $(1,1,3,6,10,15a_6',15a_7',3a_8'\cdots, 3a_{16}'),$\\
& $(1,1,3,6,12,15,25,3a_8',\cdots,3a_{16}')$\\
& $(1,1,3,5a_4',\cdots, 5a_{16}'),  $  \\
& $(1,2,2,5a_4',\cdots, 5a_{16}') $  \\
  \hline 

\rule[-2.4mm]{0mm}{7mm}  $A'(3)$       &  $ (1,1,3a_3',\cdots, 3a_{16}'),$   \\
& $(1,1,a_3,\cdots, a_{16})$ with $\{a_3,\cdots,a_{16} \pmod3\}=\{1,0,\cdots,0\}$,\\
& $(1,2,3a_3',\cdots, 3a_{16}'),$ \\
& $(1,2,3,6,8,8a_6',24,8a_8',\cdots, 8a_{16}')$ \\

  \hline

\rule[-2.4mm]{0mm}{7mm}  $A'(2)$       
& $(1,2,2,3,8a_5',\cdots,8a_{16}'),$\\
&  $ (1,2,2,5,8a_5',\cdots, 8a_{16}'), $   \\
& $(1,2,2,5,10,16a_5',\cdots, 16a_{16}') ,$\\
& $ (1,2,3,6,8a_5',\cdots, 8a_{16}'),$\\
& $(1,2,4,4,5,16a_6',\cdots, 16a_{16}'),$\\
& $(1,2,4,4,7,16a_5',\cdots, 16a_{16}'),$\\
& $(1,2,4,4,9,16a_5',\cdots, 16a_{16}'), $\\
 & $(1,2,4,4,11,16a_5',\cdots, 16a_{16}'),$\\ 
 & $ (1,2,4,5,12,16a_5',\cdots, 16a_{16}'), $\\ 
& $(1,2,4,7,12,16a_5',\cdots, 16a_{16}') $\\

  \hline
\end{tabular}
\end{table}
\end{center}

\section{Generalization}
In Section 5 and Section 6, we show that for each $\mathbf a \in A'$, there is a constant $C_{\mathbf a}>0$ for which the universality of any $m$-gonal form having its first $16$ coefficients as $\mathbf a $ is characterized by the representability of every positive integer up to $C_{\mathbf a}(m-2)$.
As a first step to deal with the dropouts $\mathbf a \in A'$, we suggest an improvement of Lemma \ref{lu}.
Under the same assumptions and notations used in Lemma \ref{lu}, for each $0\le t_1 \le s+1$ and $-1 \le r_1 \le 2^{16}-1$, suppose that there are $10$-subtuple $\mathbf a_2$ of $\mathbf a$ and $\beta(t_1,r_1) \in \Z^{10}$ with $B_{\mathbf a_2}(\alpha, \beta(t_1,r_1))=1$ for which $$Q_{\mathbf a_2}(\mathbf x(t_1,r_1))+r_1^2Q_{\mathbf a_2}(\beta(t_1,r_1))-r_1 \equiv 2t_1 \pmod{2s}$$
for some $\mathbf x(t_1,r_1)\in L_{\mathbf a_2,\beta(t_1,r_1)}$. 
Then the $10$-ary sub $m$-gonal form $\Delta_{m,\mathbf a_2}(\mathbf x)$ of $\Delta_{m,\mathbf a}(\mathbf x)$ may represent all the $t_1(m-2)+r_1$'s modulo $s(m-2)$ in the interval 
$$\left[0, \ \frac{m-2}{2} \cdot \max \limits_{(t_1,r_1)}\{Q_{\mathbf a_2}(\mathbf x(t_1,r_1))+r_1^2Q_{\mathbf a_2}(\beta(t_1,r_1))-r_1 \}+(m-2)\right]$$
for 
$$\Delta_{m,\mathbf a_2}(\mathbf x(t_1,r_1)+r_1\beta(t_1,r_1))=\frac{m-2}{2}\left\{Q_{\mathbf a_2}(\mathbf x(t_1,r_1))+r_1^2Q_{\mathbf a_2}(\beta(t_1,r_1))-r_1\right\}+r_1.$$
Actually, in Lemma \ref{lu}, the sole vector $\beta \in \Z^{10}$ fulfil the all of duties for  all $\beta(t_1, r_1)\in \Z^{10}$ where $0\le t_1 \le s+1$ and $-1 \le r_1 \le 2^{16}-1$.

In the following lemma, as a kind of generalization of Lemma \ref{lu}, we suggest a sufficient local condition to exist such the above $\beta(t_1, r_1) \in \Z^{10}$'s for each $0\le t_1 \le s+1$ and $-1 \le r_1 \le 2^{16}-1$.

\begin{lem} \label{lu'}
For a prime $p_0$, if there are a $10$-subtuple $\mathbf a_2:=(a_{i_1},\cdots,a_{i_{10}})$ of $(a_1,\cdots,a_{16}) \in A$ and $N$ vectors $\beta^{(k)}\in \mathbb Z^{10}$ with $B_{\mathbf a_2}(\alpha,\beta^{(k)})=1$ for all $1\le k \le N$ satisfying the following two local conditions

\begin{equation} \label{2-1}
L_{\mathbf a_2, \beta^{(k)}}\otimes \mathbb Z_p \text{ are even universal for all primes $p\not=p_0$}  
\end{equation}
for all $ 1\le k \le N$ and
\begin{equation} \label{2-2}
\bigcup\limits_{k=1}^N\{Q_{\mathbf a_2}(\mathbf x)+r_1^2\cdot Q_{\mathbf a_2}(\beta^{(k)})^2-r_1 |\mathbf x \in L_{\mathbf a_2, \beta^{(k)}} \otimes \mathbb Z_{p_0} \} =2 \mathbb Z_{p_0} 
\end{equation}
for each $r_1 \in \z_{p_0}$, then there is a constant $$C_{(a_1,\cdots, a_{16})}>0$$
for which having its first $16$ coefficients as $(a_1,\cdots ,a_{16}) \in A$ $m$-gonal form with $m>2^{16}-3$ $$\Delta_{m,\mathbf a}(\mathbf x)= \sum\limits_{i=1}^n a_iP_m(x_i)$$ which represents every positive integer up to $m-4$ represents every positive integer greater than $C_{(a_1,\cdots,a_{16})}(m-2)$.
\end{lem}

\begin{proof}
For each $0 \le t_1 \le s+1$ and $-1\le r_1 \le 2^{16}-1$, 
$$Q_{\mathbf a_2}(\mathbf x)+r_1^2Q_{\mathbf a_2}(\beta(t_1,r_1))-r_1 \equiv 2t_1 \pmod{2s}$$ has an integer solution $\mathbf x(t_1,r_1) \in \z^n$ for some $\beta(t_1,r_1) \in \{\beta^{(1)},\cdots, \beta^{(N)}\}$ by Theorem \ref{hkk}. 
Which induces that for any $0 \le t_1 \le s+1$ and $-1\le r_1 \le 2^{16}-1$,
$$\begin{cases}\Delta_{m,\mathbf a_2}(\mathbf x(t_1,r_r)+r_1\beta(t_1,r_1))\equiv t_1(m-2)+r_1 \pmod{s(m-2)}\\
0\le \Delta_{m,\mathbf a_2}(\mathbf x(t_1,r_r)+r_1\beta(t_1,r_1)) \le C_{(a_1,\cdots,a_{16})}(m-2)\end{cases}$$  
hold where $$C_{(a_1,\cdots,a_{16})}:=\frac{1}{2}\max \limits_{(t_1,r_1)}\{Q_{\mathbf a_2}(\mathbf x(t_1,r_1))+r_1^2Q_{\mathbf a_2}(\beta(t_1,r_1))-r_1 \}+1.$$
\end{proof}

However, the essential problem for the dropouts $(a_1,\cdots,a_{16}) \in A'$ is for any its $10$-subtuple $\mathbf a_2$, there is 
$0 \le t_1 \le s+1$ and $-1\le r_1 \le2^{16}-1$ for which there is no the above $\beta(t_1,r_1) \in \Z^{10}$. 
There is a reason for it.
From now on, we deal with the problem and show the existence of a constant $C_{(a_1,\cdots,a_{16})}>0$ for each $(a_1,\cdots,a_{16}) \in A'$ for which the universality of $m$-gonal form $\Delta_{m,\mathbf a}(\mathbf x)$ having first $16$ coefficients as $(a_1,\cdots,a_{16})$ is characterized as the representability of positive integer up to $C_{(a_1,\cdots,a_{16})}(m-2)$ to complete the proof of Theorem \ref{main}. 

\vskip 0.8em

For $m\ge 2^{17}s-2s+2>2^{16}+3$, suppose that  an $m$-gonal form $\Delta_{m,\mathbf a}(\mathbf x)=\sum_{i=1}^na_iP_m(x_i)$ represents every positive integer up to $s(m-2)$.
For each representative $0 \le t(m-2)+r <s(m-2)$ modulo $s(m-2)$, we may take a vector $\mathbf x(t,r)=(x_1(t,r),\cdots, x_n(t,r)) \in \mathbb Z^n$ which represents $t(m-2)+r$, i.e., $\Delta_{m,\mathbf a}(\mathbf x (t,r))=t(m-2)+r$.
On the other hand, some 6 components among the first 16 components of $\Delta_{m,\mathbf a}(\mathbf x)$ are used to represent the multiples of $s(m-2)$.
So we desire that the remaining 10 components among the first 16 components handle the absent of the 6 components.
For each $0 \le t < s$ and $0 \le r <m-2$, write $$\sum \limits_{i=1}^{16}a_iP_m(x_i(t,r))=:t_1(m-2)+r_1 \ \text{ where }\left[-\frac{m}{2}\right] \le r_1 \le \left[\frac{m-2}{2}\right].$$
Note that $|x_i(t,r)|<s$ for all $1 \le i \le 16$ since $\sum \limits_{i=1}^{16}a_iP_m(x_i(t,r)) < s(m-2)$.
In fact, the absolute values are may much smaller than $s$ (possibly about $\sqrt{2s}$), but in here our point is that the possible $x_i(t,r)$'s are finitely many independently on $m$.
From $\left[-\frac{m}{2}\right] \le r_1 \le \left[\frac{m-2}{2}\right]$, we have $$r_1=\sum \limits_{i=1}^{16}a_i \cdot x_i(t,r)$$ since $|r_1|\le \sum_{i=1}^{16}a_i \cdot |x_i(t,r)| <s(a_1+\cdots+a_{16}) \le s(2^{16}-1)$.
Thus regardless of $m$, we have 
\begin{equation}\label{tr}\begin{cases} 0\le t_1 \le s+1 \\ |r_1|<s(2^{16}-1). \end{cases}\end{equation}
For the convenience of notation, for $\mathbf a=(a_1,\cdots,a_{16}) \in A'$, we define a set 
$$S_{\mathbf a,r_1}:=\{t'\in \mathbb N_0| t'(m-2)+r_1= \sum \limits_{i=1}^{16}a_iP_m(x_i) \le s(m-2) \text{ for some } x_i \in \mathbb Z\}$$
by the set of all the possible $t_1$'s for each $\left[-\frac{m}{2}\right] \le r_1 \le \left[\frac{m-2}{2}\right]$ and a quadratic polynomial $$Q_{\mathbf a,r_1,\beta^{(k)}}(\mathbf x):=Q_{\mathbf a}(\mathbf x)+r_1^2\cdot Q_{\mathbf a}(\beta^{(k)})^2-r_1.$$
By (\ref{tr}), we see that $$\begin{cases}S_{\mathbf a,r_1}\text{$\not= \emptyset$ only if $|r_1|\le s(2^{16}-1)$} \\ \text{$S_{\mathbf a,r_1} \subseteq \{n \in \mathbb N_0| 0 \le n \le s+1\}$}. \end{cases}$$ 
In order to show that for $\mathbf a \in A'$, there is a constant $C_{\mathbf a}>0$ for which the universality of any $m$-gonal form having its first $16$ coefficients as $\mathbf a $ is characterized by the representability of every positive integer up to $C_{\mathbf a}(m-2)$, it would be enough to show that there is a $10$-subtuple $\mathbf a_2$ of $\mathbf a$ for which the $10$-ary $m$-gonal form $\Delta_{m,\mathbf a_2}(\mathbf x)$ represents an integer which is equivalent with $t_1(m-2)+r_1$ modulo $s(m-2)$ in an interval $[0, (C_{\mathbf a}-s)(m-2)]$ for each $t_1 \in S_{\mathbf a,r_1}$ and $|r_1|<s(2^{16}-1)$.

\vskip 0.8em

Now our goal is to show that for each $\mathbf a \in A'$, there is a $10$-subtuple $\mathbf a_2$ of $\mathbf a$ for which the $10$-ary $m$-gonal form $\Delta_{m,\mathbf a_2}(\mathbf x)$ represents an integer which is equivalent with $t_1(m-2)+r_1$ modulo $s(m-2)$ for all $t_1 \in S_{\mathbf a,r_1}$ and $|r_1|<s(2^{16}-1)$ in an interval $[0, (C_{\mathbf a}-s)(m-2)]$ by showing that there is $\beta(t_1,r_1) \in \Z^{10}$ for each $t_1 \in S_{\mathbf a,r_1}$ and $|r_1|<s(2^{16}-1)$.
And then we may obtain that for $\Delta_{m,\mathbf a_2}(\mathbf y (t_1,r_1))\equiv t_1(m-2)+r_1 \pmod{s(m-2)}$ with $0\le \Delta_{m,\mathbf a_2}(\mathbf y (t_1,r_1))\le (C_{(a_1,\cdots,a_{16})}-s)(m-2)$, 
$$\begin{cases}
\Delta_{m,\mathbf a_2}(\mathbf y(t_1,r_1))+\sum \limits_{i=17}^na_iP_m(x_i(t,r))\equiv t(m-2)+r \pmod{s(m-2)} \\
\Delta_{m,\mathbf a_2}(\mathbf y(t_1,r_1))+\sum \limits_{i=17}^na_iP_m(x_i(t,r))\le C_{(a_1,\cdots,a_{16})}(m-2),\end{cases}$$ then we are done.

In Lemma \ref{3}, we suggest a sufficient local condition to exist such the above $\beta(t_1, r_1) \in \Z^{10}$ for all $t_1 \in S_{\mathbf a,r_1}$ and $|r_1|<s(2^{16}-1)$.

\vskip 0.8em

\begin{lem} \label{3}
For a prime $p_0$, if there are a $10$-subtuple $\mathbf a_2:=(a_{i_1},\cdots,a_{i_{10}})$ of $(a_1,\cdots,a_{16}) \in A'$ and $N$ vectors $\beta^{(k)} \in \mathbb Z^{10}$ with $B_{\mathbf a_2}(\alpha,\beta^{(k)})=1$ for all $1\le k \le N$ satisfying the following two conditions
\begin{equation} \label{t1}
L_{\mathbf a_2, \beta^{(k)}}\otimes \mathbb Z_p \text{ are even universal for all primes $p\neq p_0$  }   
\end{equation} 
for all $1\le k \le N$ and
\begin{equation} \label{t2}
\bigcup\limits_{k=1}^N\{Q_{\mathbf a_2,r_1,\beta^{(k)}}(\mathbf x) |\mathbf x \in L_{\mathbf a_2, \beta^{(k)}} \otimes \mathbb Z_{p_0} \} \supseteq 2S_{(a_1,\cdots,a_{16}), r_1} \otimes \mathbb Z_{p_0} ,
\end{equation}
for each $r_1 \in \mathbb Z_{p_0}$, then there is 
$$C_{(a_1,\cdots, a_{16})} \ge s>0$$ 
for which having its first $16$ coefficients as $(a_1,\cdots ,a_{16}) \in A'$ $m$-gonal form with $m\ge 2^{17}s-2s+2>2^{16}+3$ 
$$\Delta_{m,\mathbf a}(\mathbf x)= \sum\limits_{i=1}^n a_iP_m(x_i)$$ which represents every positive integer up to $s(m-2)$ represents every positive integer greater than $C_{(a_1,\cdots,a_{16})}(m-2)$.
\end{lem}

\begin{proof}
For each representative $0 \le t(m-2)+r <s(m-2)$, take a vector $\mathbf x(t,r)=(x_1(t,r),\cdots, x_n(t,r)) \in \mathbb Z^n$ such that 
$$\Delta_{m,\mathbf a}(\mathbf x(t,r))=t(m-2)+r.$$
We write $$\sum \limits_{i=1}^{16}a_iP_m(x_i(t,r))=:t_1(m-2)+r_1$$ where $\left[-\frac{m}{2} \right] \le r_1 \le \left[\frac{m-2}{2} \right]$.
Since $t_1 \in S_{(a_1,\cdots,a_{16}),r_1}$, we may take a vector $\beta(t_1,r_1) \in \{\beta^{(1)},\cdots, \beta^{(N)}\}$ for which 
$$2t_1 \in \{Q_{\mathbf a_2}(\mathbf x)+r_1^2\cdot Q_{\mathbf a_2}(\beta(t_1,r_1))^2-r_1 |\mathbf x \in L_{\mathbf a_2, \beta(t_1,r_1)} \otimes \mathbb Z_{p_0} \}.$$
On the other hand, since $L_{\mathbf a_2,\beta^{(k)}}\otimes \mathbb Z_p$ are even universal for all primes $p\not= p_0$, $L_{\mathbf a_2,\beta(t_1,r_1)}$ locally represents $2t_1-r_1^2\cdot Q_{\mathbf a_2}(\beta(t_1,r_1))^2+r_1$. 
So $L_{\mathbf a_2,\beta(t_1,r_1)}$ represents an integer which is equivalent with $2t_1-r_1^2\cdot Q_{\mathbf a_2}(\beta(t_1,r_1))^2+r_1$ modulo $2s$ by Theorem \ref{hkk}.
For a vector $\mathbf y'(t_1,r_1) \in L_{\mathbf a_2, \beta(t_1,r_1)}$ with 
$$Q_{\mathbf a_2}(\mathbf y' (t_1, r_1)) \equiv 2t_1-r_1^2\cdot Q_{\mathbf a_2}(\beta(t_1,r_1))^2+r_1 \pmod{2s},$$
we have that $$\Delta_{m,\mathbf a_2}(\mathbf y'(t_1,r_1)+r_1\beta(t_1,r_1))\equiv t_1(m-2)+r_1 \pmod{s(m-2)}$$ 
and so 
$$\Delta_{m,\mathbf a_2}(\mathbf y'(t_1,r_1)+r_1\beta(t_1,r_1))+\sum \limits_{i=17}^na_iP_m(x_i(t,r))\equiv t(m-2)+r \pmod{s(m-2)}.$$
Thus we may conclude that the $(n-6)$-ary subform 
$$\sum_{k=1}^{10}a_{i_k}P_m(x_{i_k})+\sum_{i=17}^na_iP_m(x_i)$$
of $\Delta_{m,\mathbf a}(\mathbf x)$ represents complete residues modulo $s(m-2)$ in $[0,C_{(a_1,\cdots,a_{16})}(m-2)]$ where 
$$C_{(a_1,\cdots,a_{16})}:= \frac{1}{2}\max \{Q_{\mathbf a_2,r_1,\beta^{(k)}}(\mathbf y'(t_1,r_1)) | |r_1|\le s(2^{16}-1) \text{ and } 2t_1 \in S_{(a_1,\cdots,a_{16}),r_1}\}+s.$$
By Corollary \ref{6}, this completes the proof.
\end{proof}

\vskip 1em

$\bold{\ VII. \ A'(7)}$\\
Consider $\mathbf a=(a_1,\cdots,a_{16})=(1,2,4,7a_4',\cdots, 7a_{16}') \in A'(7)$. \\
From definition of $A'(7)$, we have that $$a_4'=1, \quad a_5' \in \{1,2\},\quad a_6' \in \{1,2,3,4\}, \quad a_7' \in \{1,2,\cdots, 8\}.$$
Take $$\mathbf a_2:=(a_1,\cdots,a_{10})$$ and 
$$\beta^{(1)}:=\mathbf e_1, \quad \beta^{(2)}:=-\mathbf e_1+\mathbf e_2, \quad  \beta^{(3)}:=-3\mathbf e_1+\mathbf e_3$$
 with $B_{\mathbf a_2}(\alpha,\beta^{(k)})=1$ for all $1\le k \le 3$.
By Proposition \ref{uni}, for any odd prime $p\neq 7=:p_0$, $L_{\mathbf a_2,\beta^{(k)}}\otimes \Z_p$ are (even) universal for all $1\le k \le 3$.
For $p=2$, one may show that $L_{\mathbf a_2,\beta^{(k)}}\otimes \Z_2$ are even universal for all $k=1,2,3$ through similar processings with $\bold{\ I. \ A(13)}$ or $\bold{\ III. \ A(7)}$.
So, (\ref{t1}) in Lemma \ref{3} holds for $\mathbf a_2=(a_1,\cdots,a_{10})$, $\beta^{(1)}=\mathbf e_1,  \beta^{(2)}=-\mathbf e_1+\mathbf e_2,   \beta^{(3)}=-3\mathbf e_1+\mathbf e_3$ and $p_0=7$.

Now consider $S_{\mathbf a,r_1} \otimes \Z_7$ for $|r_1|<s(2^{16}-1)$.\\
When $r_1 \equiv 0 \pmod{7}$, $S_{\mathbf a,r_1}\otimes \Z_7 \subseteq \Z_7 \backslash (\Z_7^{\times})^2\subsetneq \Z_7$.
For $t_1 \in S_{\mathbf a,r_1}$ with $$x_1(t_1,r_1)+2x_2(t_1,r_1)+4x_3(t_1,r_1)\equiv \sum_{i=1}^{16}a_iP_m(x_i)=r_1 \equiv 0 \pmod{7},$$
namely, $x_1(t_1,r_1)\equiv -2x_2(t_1,r_1)-4x_3(t_1,r_1) \pmod{7}$, 
\begin{align*}2t_1 & =\sum_{i=1}^{16}a_i(x_i(t_1,r_1)^2-x_i(t_1,r_1))\\
&\equiv x_1(t_1,r_1)^2+2x_2(t_1,r_1)^2+4x_3(t_1,r_1)^2 \\ & \equiv (-2x_2(t_1,r_1)-4x_3(t_1,r_1))^2+2x_2(t_1,r_1)^2+4x_3(t_1,r_1)^2\\   & \equiv-x(t_1,r_1)^2+2x_2(t_1,r_1)x_3(t_1,r_1)-x_3(t_1,r_1)^2 \\ &=-(x_2(t_1,r_1)-x_3(t_1,r_1))^2 \pmod{7}\end{align*}
holds.
Which yields $t_1 \notin (\Z_7^{\times})^2$ since $-2$ is a quadratic non-residue modulo $7$.
On the other hand, we may use Lemma \ref{nd} to obtain that
\begin{align*}L_{\mathbf a_2,\beta^{(k)}}\otimes \Z_7=(\Z_7\alpha+\Z_7\beta^{(k)})^{\perp} &\cong (\left<1\right>\perp\left<-1\right>)^{\perp}\\ & \cong \left<-1\right>\perp\left<a_4\right>\perp \cdots \perp \left<a_{10}\right>\\\end{align*} for all $k=1,2,3$.
Since $-1$ is a quadratic non-residue modulo $7$ and $ a_4,a_5,a_6$ are prime elements of $\z_7$, we may obtain that the represented integers by $L_{\mathbf a_2,\beta^{(k)}}\otimes \Z_7$ are 
\begin{equation}\label{124}Q_{\mathbf a_2}(L_{\mathbf a_2,\beta^{(k)}}\otimes \Z_7)=\Z_7 \backslash (\Z_7^{\times})^2\end{equation}
for all $1 \le k \le 3$.
From (\ref{124}), we obtain 
\begin{equation}
\text{for }r_1 \in \Z_7^{\times},\ \bigcup\limits_{k=1}^3\{Q_{\mathbf a_2,r_1,\beta^{(k)}}(\mathbf x) |\mathbf x \in L_{\mathbf a_2, \beta^{(k)}} \otimes \mathbb Z_{7} \} =\Z_7\supseteq 2S_{\mathbf a,r_1}\otimes \Z_7 
\end{equation}
and
\begin{equation}
\text{for } r_1 \in 7\Z_7, \ \bigcup\limits_{k=1}^3\{Q_{\mathbf a_2,r_1,\beta^{(k)}}(\mathbf x) |\mathbf x \in L_{\mathbf a_2, \beta^{(k)}} \otimes \mathbb Z_{7} \} =\Z_7 \backslash (\Z_7^{\times})^2 \supseteq 2S_{\mathbf a,r_1}\otimes \Z_7,
\end{equation}
which imply that (\ref{t2}) in Lemma \ref{3} holds for $\mathbf a_2=(a_1,\cdots,a_{10})$, $\beta^{(1)}=\mathbf e_1,  \beta^{(2)}=-\mathbf e_1+\mathbf e_2,   \beta^{(3)}=-3\mathbf e_1+\mathbf e_3$ and $p_0=7$.

Consequently, we conclude that when $\mathbf a \in A'(7)$,
for $\mathbf a_2=(a_1,\cdots,a_{10})$, $\beta^{(1)}=\mathbf e_1,  \beta^{(2)}=-\mathbf e_1+\mathbf e_2,   \beta^{(3)}=-3\mathbf e_1+\mathbf e_3$ and $p_0=7$, both of (\ref{t1}) and (\ref{t2}) in Lemma \ref{3} are satisfied.

\vskip 1em

$\bold{\ VIII. \ A'(5)}$ 

First consider $\mathbf a=(a_1,\cdots,a_{16})=(1,1,3,5a_4',\cdots,5a_{16}' ) \in A'(5)$.\\
From definition of $A'(5)$, we have that $$a_4'=1, \quad a_5' \in \{1,2\},\quad a_6' \in \{1,2,3,4\}, \quad a_7' \in \{1,2,\cdots, 8\}.$$
Take $$\mathbf a_2:=(a_1,\cdots,a_{10})$$ and 
$$\beta^{(1)}:=-4\mathbf e_2+\mathbf e_4,\  \beta^{(2(1))}:=-5\mathbf e_1-7\mathbf e_2+\mathbf e_3+\mathbf e_4+\mathbf e_5,\beta^{(2(2))}:=-7\mathbf e_1-5\mathbf e_2+\mathbf e_3+\mathbf e_5 $$
 with $B_{\mathbf a_2}(\alpha,\beta^{(k)})=1$ for all $k=1,2(1),2(2)$.
By Proposition \ref{uni}, for any odd prime $p\neq 5=:p_0$, $L_{\mathbf a_2,\beta^{(k)}}\otimes \Z_p$ is (even) universal for each $k \in \{1,2(1),2(2)\}$.
For $p=2$, one may show that $L_{\mathbf a_2,\beta^{(k)}}\otimes \Z_2$ with $a_5=5$ (resp, $a_5=10$) is even universal for each $k=1,2(1)$ (resp, $k=1,2(2)$) through a similar processing with $\bold{\ I. \ A(13)}$ or $\bold{\ III. \ A(7)}$.
So, the above arguments induce that when $\mathbf a=(a_1,\cdots,a_{16})=(1,1,3,5a_4',\cdots,5a_{16}' ) \in A'(5)$ with $a_5=5$ (resp, $a_5=10$), (\ref{t1}) in Lemma \ref{3} holds for $\mathbf a_2=(a_1,\cdots,a_{10})$, $\beta^{(1)}=-4\mathbf e_2+\mathbf e_4,\  \beta^{(2)}=-5\mathbf e_1-7\mathbf e_2+\mathbf e_3+\mathbf e_4+\mathbf e_5$ (resp, $\beta^{(1)}=-4\mathbf e_2+\mathbf e_4,\  \beta^{(2)}=-7\mathbf e_1-5\mathbf e_2+\mathbf e_3+\mathbf e_5$), and $p_0=5$.

Now consider $S_{\mathbf a,r_1} \otimes \Z_5$ for $|r_1|<s(2^{16}-1)$.\\
When $r_1 \equiv 0 \pmod{5}$, we may show that $S_{\mathbf a,r_1}\otimes \Z_5 \subseteq \Z_5 \backslash 2(\Z_5^{\times})^2\subsetneq \Z_5$ through a similar processing with $\bold{\ VII. \ A'(7)}$.
On the other hand, for $\mathbf a_2=(a_1,\cdots,a_{10})$, we may use Lemma \ref{nd} to see that
\begin{align*}L_{\mathbf a_2,\beta^{(k)}}\otimes \Z_5=(\Z_5\alpha+\Z_5\beta^{(k)})^{\perp} &\cong (\left<1\right>\perp\left<-1\right>)^{\perp}\\ & \cong \left<3\right>\perp\left<a_4\right>\perp\left<a_5\right>\perp \cdots \perp \left<a_{10}\right>\\ \end{align*} 
for all $k=1,2(1), 2(2)$.
Since $3$ is a quadratic non-residue modulo $5$ and $a_4,a_5,a_6$ are prime elements of $\z_5$, we obtain that the represented integers by $L_{\mathbf a_2,\beta^{(k)}}\otimes \Z_5$ are 
\begin{equation}\label{124'}
Q_{\mathbf a_2}(L_{\mathbf a_2,\beta^{(k)}}\otimes \Z_5)=\Z_5 \backslash (\Z_5^{\times})^2\end{equation}
for all $k=1,2(1), 2(2)$.
From (\ref{124'}), we obtain 
\begin{equation}
\text{for }r_1 \in \Z_5^{\times},\ \bigcup\limits_{k\in\{1,2(i)\}}\{Q_{\mathbf a_2,r_1,\beta^{(k)}}(\mathbf x) |\mathbf x \in L_{\mathbf a_2, \beta^{(k)}} \otimes \mathbb Z_{5} \} =\Z_5\supseteq 2S_{\mathbf a,r_1}\otimes \Z_5 
\end{equation}
for each $i=1,2$ and
\begin{equation}
\text{for } r_1 \in 5\Z_5, \ \bigcup\limits_{k\in\{1,2(i)\}}\{Q_{\mathbf a_2,r_1,\beta^{(k)}}(\mathbf x) |\mathbf x \in L_{\mathbf a_2, \beta^{(k)}} \otimes \mathbb Z_{5} \} =\Z_5 \backslash (\Z_5^{\times})^2 \supseteq 2S_{\mathbf a,r_1}\otimes \Z_5,
\end{equation}
for each $i=1,2$,
which imply that when $\mathbf a=(1,1,3,5a_4',\cdots,5a_{16}' ) \in A'(5)$ with $a_5=5$ (resp, $a_5=10$), (\ref{t2}) in Lemma \ref{3} holds for $\mathbf a_2=(a_1,\cdots,a_{10})$, $\beta^{(1)}=-4\mathbf e_2+\mathbf e_4,\  \beta^{(2)}=-5\mathbf e_1-7\mathbf e_2+\mathbf e_3+\mathbf e_4+\mathbf e_5$ (resp, $\beta^{(1)}=-4\mathbf e_2+\mathbf e_4,\  \beta^{(2)}=-7\mathbf e_1-5\mathbf e_2+\mathbf e_3+\mathbf e_5$), and $p_0=5$.

Consequently, we conclude that when $\mathbf a=(1,1,3,5a_4',\cdots,5a_{16}' ) \in A'(5)$ with $a_5=5$ (resp, $a_5=10$), for $\mathbf a_2=(a_1,\cdots,a_{10})$, $\beta^{(1)}=-4\mathbf e_2+\mathbf e_4,\  \beta^{(2)}=-5\mathbf e_1-7\mathbf e_2+\mathbf e_3+\mathbf e_4+\mathbf e_5$ (resp, $\beta^{(1)}=-4\mathbf e_2+\mathbf e_4,\  \beta^{(2)}=-7\mathbf e_1-5\mathbf e_2+\mathbf e_3+\mathbf e_5$), and $p_0=5$, both of (\ref{t1}), and (\ref{t2}) in Lemma \ref{3} are satisfied.

\vskip 0.5em

Secondly, consider $\mathbf a=(a_1,\cdots,a_{16})=(1,2,2,5a_4',\cdots, 5a_{16}') \in A'(5)$.\\
Through a similar processing with the above, one may show that both of (\ref{t1}) and (\ref{t2}) in Lemma \ref{3} are satisfied for $\mathbf a_2=(a_1,\cdots,a_{10})$, $\beta^{(1)}=-\mathbf e_1+\mathbf e_2, \ \beta^{(2)}=-\mathbf e_2 -\mathbf e_3+\mathbf e_4$, and $p_0=5$,.

\vskip 0.5em

Through a similar processing with the above argument, one may show that both of (\ref{t1}) and (\ref{t2}) in Lemma \ref{3} are satisfied \\
\begin{itemize}
\item[(1)]
when $(1,1,3,3,6,15,25,3a_8',\cdots,3a_{16}') \in A'(5)$, for \\
$\mathbf a_2=(a_1,\cdots,a_{10})$, $\beta^{(1)}=\mathbf e_1, \ \beta^{(2)}=-28\mathbf e_1+\mathbf e_2+\mathbf e_3+\mathbf e_7$, $p_0=3$,
\item[(2)]
when $(1,1,3,3,10,15,15a_7',3a_8',\cdots,3a_{16}') \in A'(5)$, for \\
$\mathbf a_2=(a_1,\cdots,a_{10})$, $\beta^{(1)}=\mathbf e_1, \ \beta^{(2)}=-28\mathbf e_1+\mathbf e_2+\mathbf e_3+\mathbf e_5+\mathbf e_6$, $p_0=3$,

\item[(3)]
when $(1,1,3,6,10,15,15a_7',3a_8',\cdots,3a_{16}') \in A'(5)$, for \\
$\mathbf a_2=(a_1,\cdots,a_{10})$, $\beta^{(1)}=-3\mathbf e_3+\mathbf e_5, \ \beta^{(2)}=-\mathbf e_1 -\mathbf e_2+6\mathbf e_3-\mathbf e_6$, $p_0=3$,
\item[(4)]
when $(1,1,3,6,6,15,25,3a_8',\cdots,3a_{16}') \in A'(5)$, for \\
$\mathbf a_2=(a_1,\cdots,a_{10})$, $\beta^{(1)}=\mathbf e_1, \ \beta^{(2)}=-34\mathbf e_1+\mathbf e_2+\mathbf e_3+\mathbf e_4+\mathbf e_6$, $p_0=3$,
\item[(5)]
when $(1,1,3,6,12,15,25,3a_8',\cdots,3a_{16}') \in A'(5)$, for \\
$\mathbf a_2=(a_1,\cdots,a_{10})$, $\beta^{(1)}=\mathbf e_1, \ \beta^{(2)}=-34\mathbf e_1+\mathbf e_2+\mathbf e_3+\mathbf e_4+\mathbf e_6$, $p_0=3$.
\end{itemize}
\vskip 1em

$\bold{\ IX. \ A'(3)}$ 

First consider $\mathbf a=(a_1,\cdots , a_{16})=(1,1,3a_3',\cdots, 3a_{16}') \in A'(3)$.\\
We have that
\begin{equation}\begin{cases}S_{\mathbf a,r_1} \subseteq \Z_3 \setminus 2(\Z_3^{\times})^2 & \text{when } r_1 \equiv 0 \pmod{3}\\
S_{\mathbf a,r_1} \subseteq \Z_3 \setminus (\Z_3^{\times})^2 & \text{when } r_1 \equiv 1 \pmod{3} \\
S_{\mathbf a,r_1} \subseteq \Z_3 \setminus 2(\Z_3^{\times})^2 & \text{when } r_1 \equiv 2 \pmod{3} \end{cases}\end{equation}
since $$2t_1\equiv x_1(t_1,r_1)^2+x_2(t_1,r_1)^2-r_1 \pmod{3}$$
and
$$\begin{cases}(x_1(t_1,r_1), x_2(t_1,r_1))\equiv (1,2),(2,1), \text{ or }(0,0) \pmod{3} & \text{when } r_1 \equiv 0 \pmod{3}\\
(x_1(t_1,r_1), x_2(t_1,r_1))\equiv (1,0),(0,1), \text{ or }(2,2) \pmod{3} & \text{when } r_1 \equiv 1 \pmod{3}\\
(x_1(t_1,r_1), x_2(t_1,r_1))\equiv (2,0),(0,2), \text{ or }(1,1) \pmod{3} & \text{when } r_1 \equiv 2 \pmod{3}\\
\end{cases}$$
for $t_1 \in S_{\mathbf a,r_1}$ and $r_1=\sum_{i=1}^{16}a_ix_i(t_1,r_1)\equiv x_1(t_1,r_1)+x_2(t_1,r_1) \pmod{3}.$
On the other hand, we have that
$$\begin{cases} L_{\mathbf a_2,\beta}\otimes \Z_p \text{ is even universal} & \text{for } p\neq 3  \\ 
Q_{\mathbf a_2}(L_{\mathbf a_2,\beta}\otimes \Z_3)=\Z_3 \setminus (\Z_3^{\times})^2 . & 
\end{cases}$$
for $\mathbf a_2=(a_1,\cdots,a_{10})$ and
\begin{equation} \label{beta3'}
\begin{cases}
\beta=-\mathbf e_1-\mathbf e_2-\mathbf e_3-\mathbf e_4-\mathbf e_5+4\mathbf e_6 & \text{when } (a_3,a_4,a_5,a_6)=(3,3,3,3)\\
\beta=-\mathbf e_1-\mathbf e_2-\mathbf e_3-\mathbf e_4-\mathbf e_5+2\mathbf e_6 & \text{when } (a_3,a_4,a_5,a_6)=(3,3,3,6)\\
\beta=-4\mathbf e_1-4\mathbf e_2-4\mathbf e_3-4\mathbf e_4-4\mathbf e_5+5\mathbf e_6 & \text{when } (a_3,a_4,a_5,a_6)=(3,3,3,9)\\
\beta=-\mathbf e_1-\mathbf e_2-\mathbf e_3-\mathbf e_4-\mathbf e_5+\mathbf e_6 & \text{when } (a_3,a_4,a_5,a_6)=(3,3,3,12)\\
\beta=-\mathbf e_1-\mathbf e_2-\mathbf e_3+4\mathbf e_4-\mathbf e_5 & \text{when } (a_3,a_4,a_5)=(3,3,6)\\
\beta=-\mathbf e_1-\mathbf e_2-\mathbf e_3-\mathbf e_4-\mathbf e_5+2\mathbf e_6 & \text{when } (a_3,a_4,a_5,a_6)=(3,3,9,9)\\
\beta=-4\mathbf e_1-4\mathbf e_2-4\mathbf e_3-4\mathbf e_4+9\mathbf e_5-4\mathbf e_6 & \text{when } (a_3,a_4,a_5,a_6)=(3,3,9,12)\\
\beta=2\mathbf e_1+2\mathbf e_2+2\mathbf e_3+2\mathbf e_4-5\mathbf e_5+2\mathbf e_6 & \text{when } (a_3,a_4,a_5,a_6)=(3,3,9,15)\\
\beta=-\mathbf e_1-\mathbf e_2-\mathbf e_3-\mathbf e_4-\mathbf e_5+\mathbf e_6 & \text{when } (a_3,a_4,a_5,a_6)=(3,3,9,18)\\
\beta=-\mathbf e_1-\mathbf e_2-\mathbf e_3-\mathbf e_4+2\mathbf e_5 & \text{when } (a_3,a_4,a_5)=(3,6,6)\\
\beta=-4\mathbf e_1-4\mathbf e_2-4\mathbf e_3-4\mathbf e_4+5\mathbf e_5 & \text{when } (a_3,a_4,a_5)=(3,6,9)\\
\beta=-\mathbf e_1-\mathbf e_2-\mathbf e_3-\mathbf e_4+\mathbf e_5 & \text{when } (a_3,a_4,a_5)=(3,6,12).\\
\end{cases}
\end{equation}

Consequently, we conclude that when $\mathbf a =(1,1,3a_3',\cdots, 3a_{16}') \in A'(3)$, for $\mathbf a_2=(a_1,\cdots,a_{10})$, $\beta \in \Z^{10}$ in (\ref{beta3'}), and $p_0=3$, the local conditions (\ref{t1}) and (\ref{t2}) in Lemma \ref{lu'} are satisfied.

\vskip 0.5em

Secondly, consider $\mathbf a=(1,2,3a_3',\cdots, 3a_{16}') \in A'(3)$.\\
One may show that
\begin{equation}S_{\mathbf a,r_1} \subseteq 3\Z_3  \text{ when } r_1 \equiv 0 \pmod{3}\end{equation}
similarly with the above since $$2t_1\equiv x_1(t_1,r_1)^2+2x_2(t_1,r_1)^2-r_1 \pmod{3}$$
where $t_1 \in S_{\mathbf a,r_1}$ and $r_1=\sum_{i=1}^{16}a_ix_i(t_1,r_1)\equiv x_1(t_1,r_1)+2x_2(t_1,r_1) \pmod{3}.$
On the other hand, one may show that 
$$\begin{cases} 
L_{\mathbf a,\beta^{(k)}}\otimes \Z_p \text{ are even universal} & \text{for }p \neq 3\\
L_{\mathbf a,\beta^{(k)}}\otimes \Z_3 \text{ represents every $3$-adic integer in } 3\Z_3 & \end{cases}$$
for $\mathbf a_2=(a_1,\cdots,a_{10})$ and the following $\beta^{(k)}$'s 
\begin{equation}\label{beta3''}\begin{cases}
\beta^{(1)}=\mathbf e_1, \beta^{(2)}=-\mathbf e_2+\mathbf e_3+2\mathbf e_4, \beta^{(3)}=-\mathbf e_2+\mathbf e_3 & \text{when }(a_4,a_5,a_6) = (3,6,6)\\
\beta^{(1)}=\mathbf e_1, \beta^{(2)}=-3(\mathbf e_1+\mathbf e_3+\mathbf e_6)+38\mathbf e_2, \beta^{(3)}=-\mathbf e_2+\mathbf e_3 & \text{when }(a_4,a_5,a_6) =(6,12,21)\\
\beta^{(1)}=\mathbf e_1, \beta^{(2)}=-\mathbf e_2+\mathbf e_3, \beta^{(3)}=-\mathbf e_2+\mathbf e_3 & \text{otherwise }\\
\end{cases}\end{equation}
yielding
\begin{equation}
\text{for }r_1 \in \Z_3^{\times},\ \bigcup\limits_{k=1}^3\{Q_{\mathbf a_2,r_1,\beta^{(k)}}(\mathbf x) |\mathbf x \in L_{\mathbf a_2, \beta^{(k)}} \otimes \mathbb Z_{3} \} =\Z_3\supseteq 2S_{\mathbf a,r_1}\otimes \Z_3 
\end{equation}
and
\begin{equation}
\text{for } r_1 \in 3\Z_3, \ \bigcup\limits_{k=1}^3\{Q_{\mathbf a_2,r_1,\beta^{(k)}}(\mathbf x) |\mathbf x \in L_{\mathbf a_2, \beta^{(k)}} \otimes \mathbb Z_{3} \}=3\Z_3 \supseteq 2S_{\mathbf a,r_1}\otimes \Z_3.
\end{equation}

Consequently, we conclude that when $\mathbf a =(1,2,3a_3',\cdots, 3a_{16}') \in A'(3)$, for $\mathbf a_2=(a_1,\cdots,a_{10})$, $\beta^{(k)} \in \Z^{10}$ in (\ref{beta3''}), and $p_0=3$, the local conditions (\ref{t1}) and (\ref{t2}) in Lemma \ref{lu'} are satisfied.

\vskip 0.5em

Thirdly, consider $(1,1,a_3,\cdots, a_{16}) \in A'(3)$ with $\{a_3,\cdots, a_{16}\pmod3\}=\{1,0,\cdots,0\}$.\\
Note that $$(a_3,a_4,a_5) \in \{(1,3,a_5), (3,3,a_5), (3,4,6), (3,4,9), (3,6,a_5)\}.$$
If $(a_3,a_4,a_5) \neq (3,4,9)$, then
\begin{equation}S_{\mathbf a,r_1} \subseteq \Z_3 \setminus 2(\Z_3^{\times})^2  \text{ for } r_1 \equiv 0 \pmod{3}.\end{equation}
Through a similar processing with the above, one may yield that for $\mathbf a_2=(a_1,\cdots,a_{10})$,  
\begin{equation}
\begin{cases}
\beta^{(1)}=\mathbf e_1, \beta^{(2)}=6\mathbf e_1-\mathbf e_2-\mathbf e_3-\mathbf e_4 & \text{when }(a_3,a_4)=(1,3)\\
\beta^{(1)}=\mathbf e_1, \beta^{(2)}=-\mathbf e_1-\mathbf e_2-\mathbf e_3+2\mathbf e_4 & \text{when }(a_3,a_4)=(3,3)\\
\beta^{(1)}=\mathbf e_1, \beta^{(2)}=\mathbf e_1+\mathbf e_2+\mathbf e_3-\mathbf e_4 & \text{when }(a_3,a_4,a_5)=(3,4,6)\\
\beta^{(1)}=\mathbf e_1, \beta^{(2)}=-\mathbf e_1-\mathbf e_2-\mathbf e_3+\mathbf e_4 & \text{when }(a_3,a_4)=(3,6),\\
\end{cases}
\end{equation}
and $p_0=3$, the local conditions (\ref{t1}) and (\ref{t2}) in Lemma \ref{lu'} are satisfied by showing that
\begin{equation}
L_{\mathbf a_2,\beta^{(k)}}\otimes \Z_p \text{ are even universal for every prime }p \neq 3,
\end{equation}
\begin{equation}
\text{for $r_1 \in \Z_3^{\times}$}, \ \bigcup\limits_{k=1}^2\{Q_{\mathbf a_2,r_1,\beta^{(k)}}(\mathbf x) |\mathbf x \in L_{\mathbf a, \beta^{(k)}} \otimes \mathbb Z_{3} \} =\Z_3\supseteq 2S_{\mathbf a_2,r_1}\otimes \Z_3 
\end{equation}
and
\begin{equation}
\text{for $r_1 \in 3\Z_3$}, \ \bigcup\limits_{k=1}^2\{Q_{\mathbf a_2,r_1,\beta^{(k)}}(\mathbf x) |\mathbf x \in L_{\mathbf a_2, \beta^{(k)}} \otimes \mathbb Z_{3} \}=\Z_3 \backslash (\Z_3^{\times})^2 \supseteq 2S_{\mathbf a,r_1}\otimes \Z_3.
\end{equation}
If $(a_3,a_4,a_5)=(3,4,9)$, then
\begin{equation}\label{349}
\begin{cases}
S_{\mathbf a,r_1} \subseteq 2(\Z_3^{\times})^2 \cup 3(\Z_3^{\times})^2 \cup 9\Z_3 & \text{ when } r_1 \equiv 0 \pmod{9} \\
S_{\mathbf a,r_1} \subseteq 2(\Z_3^{\times})^2 \cup 6(\Z_3^{\times})^2 \cup 9\Z_3 & \text{ when } r_1 \equiv 3 \pmod{9} \\
S_{\mathbf a,r_1} \subseteq 2(\Z_3^{\times})^2 \cup 3(\Z_3^{\times}) & \text{ when } r_1 \equiv 6 \pmod{9}. \\
\end{cases}
\end{equation}
So we may see that for $\mathbf a_2=(a_1,\cdots,a_{10})$, $\beta^{(1)}=\mathbf e_1,\ \beta^{(2)}=\mathbf e_1+\mathbf e_2+\mathbf e_3-\mathbf e_4, \ \beta^{(3)}=2\mathbf e_1+2\mathbf e_2+2\mathbf e_3-\mathbf e_5$, and $p_0=3$, the local conditions (\ref{t1}) and (\ref{t2}) in Lemma \ref{lu'} are satisfied by showing that
\begin{equation}
L_{\mathbf a_2,\beta^{(k)}}\otimes \Z_p \text{ are even universal for every prime }p \neq 3,
\end{equation}
\begin{equation}
\text{for $r_1 \in \Z_3^{\times}$}, \ \bigcup\limits_{k=1}^3\{Q_{\mathbf a_2,r_1,\beta^{(k)}}(\mathbf x) |\mathbf x \in L_{\mathbf a_2, \beta^{(k)}} \otimes \mathbb Z_{3} \} =\Z_3\supseteq 2S_{\mathbf a,r_1}\otimes \Z_3 
\end{equation}
and
\begin{equation}
\text{for $r_1 \in 3\Z_3$}, \ \bigcup\limits_{k=1}^3\{Q_{\mathbf a_2,r_1,\beta^{(k)}}(\mathbf x) |\mathbf x \in L_{\mathbf a_2, \beta^{(k)}} \otimes \mathbb Z_{3} \} \supseteq 2S_{\mathbf a,r_1}\otimes \Z_3.
\end{equation}

\vskip 0.5em

We postpone to consider the final case $(1,2,3,6,8,8a_6',24,8a_8',\cdots, 8a_{16}')\in A'(3)$ to next section.


\section{$A'(2)$ : Dropouts in $A(2)$}

We finally consider the dropouts in $A'(2)$ and of the form of $(1,2,3,6,8,8a_6',24,$ $8a_8',\cdots, 8a_{16}')\in A'(3)$ to complete Theorem \ref{main}.

\begin{cor} \label{2uni}
Suppose $p$ be an odd prime.

\begin{itemize}
\item[(1)]
For  $\mathbf a_2=(a_{i_1}, \cdots, a_{i_{10}})$, if there are at least 6 units of $\Z_p$ (by admitting a recursion) in $\{a_{i_1}, \cdots, a_{i_{10}}\}$, then $L_{\mathbf a_2, \beta} \otimes \mathbb Z_p$ is (even) universal for $\beta \in \Z^{10}$ with $B_{\mathbf a_2}(\alpha, \beta)=2^t$.


\item[(2)]
For  $\mathbf a_2=(a_{i_1}, \cdots, a_{i_{10}})$, if there are at least 5 units and at least 2 prime elements of $\Z_p$ (by admitting a recursion) in $\{a_{i_1}, \cdots, a_{i_{10}}\}$, then $L_{\mathbf a_2, \beta} \otimes \mathbb Z_p$ is (even) universal for $\beta \in \Z^{10}$ with $B_{\mathbf a_2}(\alpha, \beta)=2^t$.

\end{itemize}
\end{cor}
\begin{proof}
This directly follows from Proposition \ref{uni} since $$\Z_p\alpha+\Z_p\beta=\Z_p\alpha+\Z_p \left(2^{-t}\beta\right)$$ and $B_{\mathbf a}(\alpha, 2^{-t}\beta)=1$.
\end{proof}

\begin{rmk}
Note that 
\begin{equation}\label{6.1}
\Delta_{m,\mathbf a}(\mathbf x +r_1'\beta)=\frac{m-2}{2}\{Q_{\mathbf a}(\mathbf x)+r_1'^2Q_{\mathbf a}(\beta)-2^tr_1'\}+2^tr_1'\end{equation}
for $\mathbf x \in L_{\mathbf a,\beta}$ with $B_{\mathbf a}(\alpha,\beta)=2^t$.
\end{rmk}

\vskip 1em

$\bold{\ X. \ A'(2)-1}$ 

Consider $\mathbf a =(1,2,2,5,8a_5',\cdots,8a_{16}') \in A'(2)$.\\
From definition of $A'(2)$, we have that $$a_5'=1, \quad a_6' \in \{1,2\}, \quad a_7' \in \{1,2,3,4\}.$$
Let $$\mathbf a_2:=(a_1,\cdots,a_{10})$$ and
$$\beta:=\mathbf e_1+\mathbf e_2+\mathbf e_3+\mathbf e_4-\mathbf e_5 \in \Z^{10}$$
with $B_{\mathbf a_2}(\alpha, \beta)=2.$
By Corollary \ref{2uni}, $L_{\mathbf a_2,\beta}\otimes \Z_p$ is (even) universal for odd prime $p$.
By using Lemma \ref{odd2} (with $j_1=1,j_2=2,j_3=3,j_4=4$), we may obtain that $L_{\mathbf a_2,\beta}\otimes \Z_2$ is even universal.
Therefore $L_{\mathbf a_2,\beta}$ is locally even universal.
By Theorem \ref{hkk}, for each $0\le t_1 \le s-1$ and $0 \le r_1 \le 2^{16}-1$ with $r_1 \equiv 0\pmod{2}$, $L_{\mathbf a_2,\beta}$ would represent an even integer which is equivalent with $2t_1-\left(\frac{r_1}{2}\right)^2Q_{\mathbf a_2}(\beta)+r_1(\equiv 0 \pmod{2})$ modulo $2s$, i.e., there is an $\mathbf x (t_1,r_1) \in L_{\mathbf a_2,\beta}$ such that
$$Q_{\mathbf a_2}(\mathbf x (t_1,r_1)) \equiv 2t_1-\left(\frac{r_1}{2}\right)^2Q_{\mathbf a_2}(\beta)+r_1 \pmod{2s}.$$
And then from (\ref{6.1}), we obtain 
$$\Delta_{m,\mathbf a_2}\left(\mathbf x(t_1,r_1)+\frac{r_1}{2}\beta\right) \equiv t_1(m-2)+r_1 \pmod{s(m-2)}.$$
In other words, the $10$-ary $m$-gonal form $\Delta_{m,\mathbf a_2}(\mathbf x)$ represents the residues $t_1(m-2)+r_1$ modulo $s(m-2)$ for $0 \le t_1 \le s$ and $-1\le r_1 \le 2^{16}-1$ with $r_1 \equiv 0 \pmod{2}$
in $[0,C_{\mathbf a}^{\text{even}}(m-2)]$
where $C_{\mathbf a}^{\text{even}}:=\frac{1}{2}\max \{Q_{\mathbf a_2,r_1,\beta}  (\mathbf x(t_1,r_1)) | 0 \le t_1 \le s, \  -1\le r_1(\equiv 0\pmod2) \le 2^{16}-1 \}+1$.

We now consider the case $r_1$ is odd.
Let $$\mathbf a_2:=(a_1,\cdots,a_{10})$$
and
$$\beta^{(1)}:=\mathbf e_1, \ \beta^{(2)}:=-\mathbf e_1+\mathbf e_2 $$
with $B_{\mathbf a_2}(\alpha,\beta^{(k)})=1$ for $k=1,2$. 
By Corollary \ref{2uni}, $L_{\mathbf a_2,\beta^{(1)}}\otimes \z_p$ and $L_{\mathbf a_2,\beta^{(2)}}\otimes \z_p$ are (even) universal for every odd prime $p$.
By using Lemma \ref{odd2} and Lemma \ref{12}, one may show that $L_{\mathbf a_2,\beta^{(1)}}\otimes \z_2$ and $L_{\mathbf a_2,\beta^{(2)}}\otimes \z_2$ represents every $2$-adic even integer except $2$ and $10$ up to unit-square, i.e., $L_{\mathbf a_2,\beta^{(1)}}\otimes \z_2$ and $L_{\mathbf a_2,\beta^{(2)}}\otimes \z_2$  represent every $2$-adic even integer which is not equivalent to $2$ modulo $8$.
Which induces that 
$$\bigcup\limits_{k=1}^2\{Q_{\mathbf a_2,r_1,\beta^{(k)}}(\mathbf x) |\mathbf x \in L_{\mathbf a_2, \beta^{(k)}} \otimes \mathbb Z_{2} \} =2\Z_2$$ for $r_1 \in \mathbb Z_2^{\times}$.
Therefore for each $0\le t_1 \le s$ and odd $-1 \le r_1 \le 2^{16}-1$, $L_{\mathbf a_2,\beta^{(k(t_1,r_1))}}$ would represent an even integer which is equivalent with $2t_1-r_1^2Q_{\mathbf a}(\beta^{(k(t_1,r_1))})+r_1$ modulo $2s$ for some $k(t_1,r_1) \in \{1,2 \}$, i.e., there is an $\mathbf x (t_1,r_1) \in L_{\mathbf a_2,\beta^{(k(t_1,r_1))}}$ for which $$Q_{\mathbf a_2}(\mathbf x (t_1,r_1)) \equiv 2t_1-r_1^2Q_{\mathbf a}(\beta^{(k(t_1,r_1))})+r_1 \pmod{2s}.$$
And then from (\ref{6.1}), we obtain  
$$\Delta_{m,\mathbf a_2}(\mathbf x(t_1,r_1)+r_1\beta^{(k(t_1,r_1))}) \equiv t_1(m-2)+r_1 \pmod{s(m-2)}.$$
So we obtain that the $10$-ary $m$-gonal form $\Delta_{m,\mathbf a_2}(\mathbf x)$ represents the residues $t_1(m-2)+r_1$ modulo $s(m-2)$ for $0 \le t_1 \le s$ and $-1\le r_1 \le 2^{16}-1$ with $r_1 \equiv 1 \pmod{2}$
in $[0,C_{\mathbf a}^{\text{odd}}(m-2)]$
where $C_{\mathbf a}^{\text{odd}}:= \frac{1}{2}\max \{Q_{\mathbf a_2,r_1,\beta^{(k(t_1,r_1))}}  (\mathbf x(t_1,r_1)) | 0 \le t_1 \le s, \ -1\le r_1(\equiv 1 \pmod2) \le 2^{16}-1 \}+1$.

Consequently, we conclude that when $\mathbf a =(1,2,2,5,8a_5',\cdots,8a_{16}') \in A'(2)$, 
for its $10$-subtuple $\mathbf a_2=(a_1,\cdots,a_{10})$, $\Delta_{m,\mathbf a_2}$ represents every residue $t_1(m-2)+r_1$ modulo $s(m-2)$ where $0 \le t_1 \le s$ and $-1\le r_1 \le 2^{16}-1$ in $[0,C_{\mathbf a}(m-2)]$ where $C_{\mathbf a}:=\max \{C_{\mathbf a}^{\text{even}}, C_{\mathbf a}^{\text{odd}}\}$.
So we may claim that the universality of an $m$-gonal form having its first coefficients as $\mathbf a =(1,2,2,5,8a_5',\cdots,8a_{16}') \in A'(2)$ is characterized by representability of every positive integer up to $C_{\mathbf a}(m-2)$.

\vskip 0.5em

Through a similar processing with the above, one may show that 
\begin{itemize}
\item[(1)] when $\mathbf a=(1,2,2,5,10,16a_5',\cdots, 16a_{16}') \in A'(2)$, with $\beta=\mathbf e_2$, $\beta^{(1)}=\mathbf e_1$,  $\beta^{(2)}=-\mathbf e_1+\mathbf e_2$,
\item[(2)] when $\mathbf a=(1,2,4,4,7,16a_5',\cdots,16a_{16}')\in A'(2)$, with $\beta=\mathbf e_1+\mathbf e_2+\mathbf e_3+\mathbf e_4+\mathbf e_5-\mathbf e_6$, $\beta^{(1)}=\mathbf e_1$,  $\beta^{(2)}=-\mathbf e_1+\mathbf e_2$,
\item[(3)] when $\mathbf a=(1,2,4,4,11,16a_5',\cdots,16a_{16}')\in A'(2)$, with $\beta=3(\mathbf e_1+\mathbf e_2+\mathbf e_3+\mathbf e_4+\mathbf e_5)-4\mathbf e_6$, $\beta^{(1)}=\mathbf e_1$,  $\beta^{(2)}=-\mathbf e_1+\mathbf e_2$, and
\item[(4)] when $\mathbf a=(1,2,4,7,12,16a_5',\cdots,16a_{16}')\in A'(2)$, with $\beta=-3(\mathbf e_1+\mathbf e_2+\mathbf e_3+\mathbf e_4+\mathbf e_5)+5\mathbf e_6$, $\beta^{(1)}=\mathbf e_1$,  $\beta^{(2)}=-\mathbf e_1+\mathbf e_2$,
\end{itemize}
for its $10$-subtuple $\mathbf a_2:=(a_1,\cdots,a_{10})$, $\Delta_{m,\mathbf a_2}(\mathbf x)$ represents the residues $t_1(m-2)+r_1$ modulo $s(m-2)$ where $0 \le t_1 \le s$ and $-1\le r_1 \le 2^{16}-1$ in $[0,C_{\mathbf a}(m-2)]$ for some constant $C_{\mathbf a}>0$ which is dependent only on $\mathbf a$.

\vskip 0.5em
When $\mathbf a= (1,2,4,4,9,16a_5',\cdots, 16a_{16}') \in A'(2)$, for $\mathbf a_2=(a_1,\cdots,a_{10})$ and $\beta=\mathbf e_1+\mathbf e_2+\mathbf e_3+\mathbf e_4+\mathbf e_5-\mathbf e_6$ with $B_{\mathbf a_2}(\alpha,\beta)=4$, since $L_{\mathbf a_2,\beta}$ is locally even universal, we may induce that  the $10$-ary $m$-gonal form $\Delta_{m,\mathbf a_2}(\mathbf x)$ represents the residues $t_1(m-2)+r_1$ modulo $s(m-2)$ for $0 \le t_1 \le s$ and $-1\le r_1 \le 2^{16}-1$ with $r_1 \equiv 0 \pmod{4}$
in $[0,C_{\mathbf a}^{(4)}(m-2)]$
where $C_{\mathbf a}^{(4)}:=\frac{1}{2}\max \{Q_{\mathbf a_2,\frac{r_1}{4},\beta}  (\mathbf x(t_1,r_1)) | 0 \le t_1 \le s, -1\le r_1(\equiv 0 \pmod4) \le 2^{16}-1  \}+1$.

For $\beta^{(1)}=\mathbf e_1$ with $B_{\mathbf a_2}(\alpha,\beta^{(1)})=1$ and $\beta^{(2)}=\mathbf e_2$ with $B_{\mathbf a_2}(\alpha,\beta^{(2)})=2$, since $L_{\mathbf a_2,\beta^{(1)}}\otimes \z_p$ and $L_{\mathbf a_2,\beta^{(2)}}\otimes \z_p$ are (even) universal for every odd prime $p$ and
$$\{Q_{\mathbf a_2,r_1,\beta^{(1)}}(\mathbf x) |\mathbf x \in L_{\mathbf a_2, \beta^{(1)}} \otimes \mathbb Z_{2} \} \cup \{Q_{\mathbf a_2,\frac{r_1}{2},\beta^{(2)}}(\mathbf x) |\mathbf x \in L_{\mathbf a_2, \beta^{(2)}} \otimes \mathbb Z_{2} \}=2\z_2$$
for $r_1 \equiv 2 \pmod{4}$,
one may induce that  the $10$-ary $m$-gonal form $\Delta_{m,\mathbf a_2}(\mathbf x)$ represents the residues $t_1(m-2)+r_1$ modulo $s(m-2)$ for $0 \le t_1 \le s$ and $-1\le r_1 \le 2^{16}-1$ with $r_1 \equiv 2 \pmod{4}$
in $[0,C_{\mathbf a}^{(2)}(m-2)]$
where $C_{\mathbf a}^{(2)}:=\frac{1}{2}\max \{Q_{\mathbf a_2,\frac{r_1}{k(t_1,r_1)},\beta^{(k(t_1,r_1))}}  (\mathbf x(t_1,r_1)) | 0 \le t_1 \le s, \  -1\le r_1 (\equiv 2 \pmod4) \le 2^{16}-1  \}+1$.

For $\beta^{(3)}=\mathbf e_1$ with $B_{\mathbf a_2}(\alpha,\beta^{(3)})=1$ and $\beta^{(4)}=-\mathbf e_1+\mathbf e_2$ with $B_{\mathbf a_2}(\alpha,\beta^{(4)})=1$, since $L_{\mathbf a_2,\beta^{(3)}}\otimes \z_p$ and $L_{\mathbf a_2,\beta^{(4)}}\otimes \z_p$ are (even) universal for every odd prime $p$ and
$$\{Q_{\mathbf a_2,r_1,\beta^{(3)}}(\mathbf x) |\mathbf x \in L_{\mathbf a_2, \beta^{(3)}} \otimes \mathbb Z_{2} \} \cup \{Q_{\mathbf a_2,r_1,\beta^{(4)}}(\mathbf x) |\mathbf x \in L_{\mathbf a_2, \beta^{(4)}} \otimes \mathbb Z_{2} \}=2\z_2$$
for $r_1 \equiv 1 \pmod{2}$,
one may induce that  the $10$-ary $m$-gonal form $\Delta_{m,\mathbf a_2}(\mathbf x)$ represents the residues $t_1(m-2)+r_1$ modulo $s(m-2)$ for $0 \le t_1 \le s$ and $-1\le r_1 \le 2^{16}-1$ with $r_1 \equiv 1 \pmod{2}$
in $[0,C_{\mathbf a}^{\text{odd}}(m-2)]$
where $C_{\mathbf a}^{\text{odd}}:=\frac{1}{2}\max \{Q_{\mathbf a_2,r_1,\beta^{(k(t_1,r_1))}}  (\mathbf x(t_1,r_1)) | 0 \le t_1 \le s, \ -1\le r_1 \le 2^{16}-1  \ (r_1 \equiv 1 \pmod{2}) \}+1$.

Consequently, we conclude that when $\mathbf a =(1,2,4,4,9,16a_5',\cdots,16a_{16}') \in A'(2)$, 
for its $10$-subtuple $\mathbf a_2=(a_1,\cdots,a_{10})$, $\Delta_{m,\mathbf a_2}$ represents the residues $t_1(m-2)+r_1$ modulo $s(m-2)$ where $0 \le t_1 \le s$ and $-1\le r_1 \le 2^{16}-1$ in $[0,C_{\mathbf a}(m-2)]$ where $C_{\mathbf a}:=\max \{C_{\mathbf a}^{(4)},C_{\mathbf a}^{(2)}, C_{\mathbf a}^{\text{odd}}\}$.
So we may claim that the universality of an $m$-gonal form having its first coefficients as $\mathbf a=(1,2,4,4,9,16a_5',\cdots, 16a_{16}') \in A'(2)$ is characterized by representability of every positive integer up to $C_{\mathbf a}(m-2)$.

\vskip 1em

\begin{rmk}

On the other hand, for the remainings
\begin{align*}\mathbf a \in \{ &(1,2,3,6,8a_5',\cdots, 8a_{16}'),(1,2,4,5,12,16a_5',\cdots, 16a_{16}'), \\
 & (1,2,2,3,8a_5',\cdots,8a_{16}'),(1,2,4,4,5,16a_6',\cdots, 16a_{16}') \} \subset A'(2), \end{align*}
there occur an issue which is similar one occured when $\mathbf a  \in A'(p) $ in Section 5.
We treat the above $\mathbf a$'s in $\bold{\ XI. \ A'(2)-2}$.
\end{rmk}

\vskip 1em

$\bold{\ XI. \ A'(2)-2}$ 

Consider $\mathbf a=(1,2,2,3,8a_5',\cdots,8a_{16}') \in A'(2)$.\\
From definition of $A'(2)$, we have that $$a_5'=1, \quad a_6' \in \{1,2\}, \quad a_7' \in \{1,2,3,4\}.$$
First, note that $S_{\mathbf a,r_1} \supseteq \mathbb N_0 \setminus 1+\frac{r_1}{2}+4\mathbb N_0$ when $r_1 \equiv 0 \pmod{4}$. 
For $\mathbf a_2=(a_1,\cdots,a_{10})$ and $\beta=\mathbf e_1$ with $B_{\mathbf a_2}(\alpha,\beta)=1$, since $L_{\mathbf a_2,\beta}$ locally represents every even integer in $2\z \setminus 2+8\z$, we may induce that  the $10$-ary $m$-gonal form $\Delta_{m,\mathbf a_2}(\mathbf x)$ represents the residues $t_1(m-2)+r_1$ modulo $s(m-2)$ for $t_1 \in S_{\mathbf a,r_1}$ and $-s(2^{16}-1)\le r_1 \le s(2^{16}-1)$ with $r_1 \equiv 0 \pmod{4}$
in $[0,C_{\mathbf a}^{(4)}(m-2)]$
where $C_{\mathbf a}^{(4)}:=\frac{1}{2}\max \{Q_{\mathbf a_2,r_1,\beta}  (\mathbf x(t_1,r_1)) | t_1 \in S_{\mathbf a,r_1}, \ -s(2^{16}-1)\le r_1(\equiv 0 \pmod{4}) \le s(2^{16}-1) \}+s$.

For $\beta^{(1)}=\mathbf e_1$ with $B_{\mathbf a_2}(\alpha,\beta^{(1)})=1$ and $\beta^{(2)}=\mathbf e_2$ with $B_{\mathbf a_2}(\alpha,\beta^{(2)})=2$, since $L_{\mathbf a_2,\beta^{(1)}}\otimes \z_p$ and $L_{\mathbf a_2,\beta^{(2)}}\otimes \z_p$ are (even) universal for every odd prime $p$ and
$$\{Q_{\mathbf a_2,r_1,\beta^{(1)}}(\mathbf x) |\mathbf x \in L_{\mathbf a_2, \beta^{(1)}} \otimes \mathbb Z_{2} \} \cup \{Q_{\mathbf a_2,\frac{r_1}{2},\beta^{(2)}}(\mathbf x) |\mathbf x \in L_{\mathbf a_2, \beta^{(2)}} \otimes \mathbb Z_{2} \}=2\z_2$$
for $r_1 \equiv 2 \pmod{4}$,
one may induce that  the $10$-ary $m$-gonal form $\Delta_{m,\mathbf a_2}(\mathbf x)$ represents the residues $t_1(m-2)+r_1$ modulo $s(m-2)$ for $t_1 \in S_{\mathbf a,r_1}$ and $-s(2^{16}-1)\le r_1 \le s(2^{16}-1)$ with $r_1 \equiv 2 \pmod{4}$
in $[0,C_{\mathbf a}^{(2)}(m-2)]$
where $C_{\mathbf a}^{(2)}:=\frac{1}{2}\max \{Q_{\mathbf a_2,\frac{r_1}{k(t_1,r_1)},\beta^{(k(t_1,r_1))}}  (\mathbf x(t_1,r_1)) | t_1 \in S_{\mathbf a,r_1}, \ -s(2^{16}-1)\le r_1(\equiv 2 \pmod{4}) \le s(2^{16}-1)  \}+s$.

For $\beta^{(3)}=\mathbf e_1$ with $B_{\mathbf a_2}(\alpha,\beta^{(3)})=1$ and $\beta^{(4)}=-\mathbf e_1+\mathbf e_2$ with $B_{\mathbf a_2}(\alpha,\beta^{(4)})=1$, since $L_{\mathbf a_2,\beta^{(3)}}\otimes \z_p$ and $L_{\mathbf a_2,\beta^{(4)}}\otimes \z_p$ are (even) universal for every odd prime $p$ and
$$\{Q_{\mathbf a_2,r_1,\beta^{(3)}}(\mathbf x) |\mathbf x \in L_{\mathbf a_2, \beta^{(3)}} \otimes \mathbb Z_{2} \} \cup \{Q_{\mathbf a_2,r_1,\beta^{(4)}}(\mathbf x) |\mathbf x \in L_{\mathbf a_2, \beta^{(4)}} \otimes \mathbb Z_{2} \}=2\z_2$$
for $r_1 \equiv 1 \pmod{2}$,
one may induce that  the $10$-ary $m$-gonal form $\Delta_{m,\mathbf a_2}(\mathbf x)$ represents the residues $t_1(m-2)+r_1$ modulo $s(m-2)$ for $t_1 \in S_{\mathbf a,r_1}$ and $-s(2^{16}-1)\le r_1 \le s(2^{16}-1)$ with $r_1 \equiv 1 \pmod{2}$
in $[0,C_{\mathbf a}^{\text{odd}}(m-2)]$
where $C_{\mathbf a}^{\text{odd}}:=\frac{1}{2}\max \{Q_{\mathbf a_2,r_1,\beta^{(k(t_1,r_1))}}  (\mathbf x(t_1,r_1)) |  t_1 \in S_{\mathbf a,r_1}, \ -s(2^{16}-1)\le r_1 ( \equiv 1 \pmod{2}) \le s(2^{16}-1) \  \}+s$.

Consequently, we conclude that when $\mathbf a =(1,2,2,3,8a_5',\cdots,8a_{16}') \in A'(2)$, 
for its $10$-subtuple $\mathbf a_2=(a_1,\cdots,a_{10})$, $\Delta_{m,\mathbf a_2}$ represents the residues $t_1(m-2)+r_1$ modulo $s(m-2)$ where $t_1 \in S_{\mathbf a,r_1}$ and $-s(2^{16}-1)\le r_1 \le s(2^{16}-1)$ in $[0,C_{\mathbf a}(m-2)]$ where $C_{\mathbf a}:=\max \{ C_{\mathbf a}^{(4)},C_{\mathbf a}^{(2)}, C_{\mathbf a}^{\text{odd}}\}$.
So we may claim that the universality of an $m$-gonal form having its first coefficients as $\mathbf a=(1,2,2,3,8a_5',\cdots,8a_{16}') \in A'(2)$ is characterized by representability of every positive integer up to $C_{\mathbf a}(m-2)$.
\vskip0.5em

For $\mathbf a = (1,2,3,6,8a_5',\cdots, 8a_{16}') \in A'(2)$, we have that $S_{\mathbf a, r_1} \supseteq \mathbb N_0 \setminus -\frac{r_1}{2} + 4\mathbb N_0$ when $r_1 \equiv 2 \pmod{4}$.
Through similar arguments with the above, with
\begin{itemize}
\item[(1)] when $r_1 \equiv 0 \pmod{4}$, $\beta=\mathbf e_1+\mathbf e_2+\mathbf e_3+\mathbf e_4-\frac{2}{a_6'}\mathbf e_5$
\item[(2)] when $r_1 \equiv 2 \pmod{4}$, $\beta=\mathbf e_1$ 
\item[(3)] when $r_1 \equiv 1 \pmod{2}$, $\beta^{(1)}=\mathbf e_1$ and $\beta^{(2)}=-\mathbf e_1+\mathbf e_2$
\end{itemize}
for its $10$-subtuple $\mathbf a_2=(a_1,\cdots,a_{10})$, $\Delta_{m,\mathbf a_2}$ represents the residues $t_1(m-2)+r_1$ modulo $s(m-2)$ where $t_1 \in S_{\mathbf a,r_1}$ and $-s(2^{16}-1)\le r_1 \le s(2^{16}-1)$ in $[0,C_{\mathbf a}(m-2)]$ for some constant $C_{\mathbf a}>0$ which is dependent only on $\mathbf a$.

\vskip 0.5em

For $\mathbf a =(1,2,4,5,12,16a_5',\cdots, 16a_{16}')\in A'(2)$, we have that $S_{\mathbf a, r_1} \supseteq \mathbb N_0 \setminus -1-\frac{r_1}{2}+8\mathbb N_0$ when $r_1 \equiv 4 \pmod{8}$.
Through similar arguments with the above, with
\begin{itemize}
\item[(1)] when $r_1 \equiv 0 \pmod{8}$, $\beta=\mathbf e_1+\mathbf e_2+\mathbf e_3+\mathbf e_4+\mathbf e_5-\mathbf e_6$
\item[(2)] when $r_1 \equiv 4 \pmod{8}$, $\beta=\mathbf e_3$ 
\item[(3)] when $r_1 \equiv 2 \pmod{4}$, $\beta^{(3)}=\mathbf e_1$ and $\beta^{(4)}=\mathbf e_2$ 
\item[(4)] when $r_1 \equiv 1 \pmod{2}$, $\beta^{(5)}=\mathbf e_1$ and $\beta^{(6)}=-\mathbf e_1+\mathbf e_2$
\end{itemize}
for its $10$-subtuple $\mathbf a_2=(a_1,\cdots,a_{10})$, $\Delta_{m,\mathbf a_2}$ represents the residues $t_1(m-2)+r_1$ modulo $s(m-2)$ where $t_1 \in S_{\mathbf a,r_1}$ and $-s(2^{16}-1)\le r_1 \le s(2^{16}-1)$ in $[0,C_{\mathbf a}(m-2)]$ for some constant $C_{\mathbf a}>0$ which is dependent only on $\mathbf a$.

\vskip0.5em

For $\mathbf a = (1,2,4,4,5,16a_6',\cdots, 16a_{16}') \in A'(2)$, we have that $S_{\mathbf a, r_1} \supseteq \mathbb N_0 \setminus 1 -\frac{r_1}{2} + 8\mathbb N_0$ when $r_1 \equiv 0 \pmod{8}$.
Through similar arguments with the above, with
\begin{itemize}
\item[(1)] when $r_1 \equiv 0 \pmod{8}$, $\beta=\mathbf e_1$
\item[(2)] when $r_1 \equiv 4 \pmod{8}$, $\beta^{(1)}=\mathbf e_1$ and $\beta^{(2)}=\mathbf e_3$ 
\item[(3)] when $r_1 \equiv 2 \pmod{4}$, $\beta^{(3)}=\mathbf e_1$ and $\beta^{(4)}=\mathbf e_2$ 
\item[(4)] when $r_1 \equiv 1 \pmod{2}$, $\beta^{(5)}=\mathbf e_1$ and $\beta^{(6)}=-\mathbf e_1+\mathbf e_2$
\end{itemize}
for its $10$-subtuple $\mathbf a_2=(a_1,\cdots,a_{10})$, $\Delta_{m,\mathbf a_2}$ represents the residues $t_1(m-2)+r_1$ modulo $s(m-2)$ where $t_1 \in S_{\mathbf a,r_1}$ and $-s(2^{16}-1)\le r_1 \le s(2^{16}-1)$ in $[0,C_{\mathbf a}(m-2)]$ for some constant $C_{\mathbf a}>0$ which is dependent only on $\mathbf a$.

\end{document}